\documentclass[12pt]{article}

\newlength{\defbaselineskip}
\usepackage[a4paper , total={8.5in, 11in}, margin=1in]{geometry}

\usepackage{hyperref}
\hypersetup{
     colorlinks   = true,
     linkcolor    = blue,
     citecolor    = green
}

\usepackage{graphicx}
\usepackage{subfigure}
\usepackage{booktabs} 
\usepackage{epstopdf}
\usepackage[all]{xy}
\usepackage{mathptmx}      
\usepackage{amssymb,mathrsfs,amsmath,amsthm,amscd,amsfonts,amssymb,bm,color}
\usepackage[mathscr]{euscript}
\usepackage{stackrel}
\usepackage{fancyhdr}
\usepackage{makeidx}
\usepackage{afterpage}
\usepackage{nicefrac}       
\usepackage{microtype}      
\usepackage{mathtools}
\usepackage{multirow}
\usepackage{hyperref}       
\usepackage{url}            
\usepackage{nicefrac}       
\usepackage{mathtools}
\usepackage{pifont}
\usepackage{mathtools}

\DeclarePairedDelimiter{\dotp}{\langle}{\rangle}

\newcommand{\nonl}{\renewcommand{\nl}{\let\nl\oldnl}}

\def\1{\bm{1}}

\newcommand{\dom}{\textnormal{dom}\,}
\newcommand{\epi}{\textnormal{epi}\,}
\newcommand{\R}{\mathbb{R}}
\newcommand{\EE}{\mathbb{E}}
\newcommand{\RR}{\mathbb{R}}
\newcommand{\cX}{\mathcal{X}}

\newcommand{\bbe}{\mathbb{E}}

\newtheorem{thm}{Theorem}

\newtheorem{defn}[thm]{Definition}
\newtheorem{cor}[thm]{Corollary}

\newtheorem{rem}[thm]{Remark}
\newtheorem{lem}[thm]{Lemma}

\newtheorem{assum}{Assumption}

\usepackage[linesnumbered,ruled,vlined,commentsnumbered]{algorithm2e}
\usepackage{setspace}

\usepackage{caption}
\usepackage[hang,flushmargin]{footmisc}
\usepackage{lipsum}
\makeatletter
\newcommand{\algorithmfootnote}[2][\footnotesize]{%
  \let\old@algocf@finish\@algocf@finish
  \def\@algocf@finish{\old@algocf@finish
    \leavevmode\rlap{\begin{minipage}{\linewidth}
    #1#2
    \end{minipage}}%
  }%
}
\usepackage[T1]{fontenc}
\usepackage{etoolbox}
\makeatletter
\patchcmd{\@maketitle}{\LARGE}{\Large}{}{}
\makeatother

\providecommand{\keywords}[1]
{
  \small	
  \textbf{\textit{Keywords---}} #1
}

\begin{document}

\title{Adaptive First-and Zeroth-order Methods
for Weakly Convex Stochastic Optimization Problems
}

	\vspace{0.8cm}
\author{{Parvin Nazari}\thanks{Department of Mathematics \& Computer Science, Amirkabir University of Technology, Email: \texttt{p$\_$nazari@aut.ac.ir}}\and{Davoud Ataee Tarzanagh }\thanks{ Department of Mathematics \& UF Informatics Institute, University of Florida, Email: \texttt{tarzanagh@ufl.edu}
}
\and{George Michailidis}\thanks{ Department of Statistics \& UF Informatics Institute, University of Florida, Email: \texttt{gmichail@ufl.edu}
}
}

\date{}
	\maketitle
	
\title{}
	\maketitle

\begin{abstract}
In this paper, we design and analyze a new family of adaptive subgradient methods for solving an important class of weakly convex (possibly nonsmooth) stochastic optimization problems. Adaptive methods which use exponential moving averages of past gradients to update search directions and learning rates have recently attracted a lot of attention for solving optimization problems that arise in machine learning. Nevertheless, their convergence analysis almost exclusively requires smoothness and/or convexity of the objective function. In contrast, we establish non-asymptotic rates of convergence of first and zeroth-order adaptive methods and their proximal variants for a reasonably broad class of nonsmooth \& nonconvex optimization problems. Experimental results indicate how the proposed algorithms empirically outperform stochastic gradient descent and its zeroth-order variant for solving such optimization problems.
\end{abstract}

\keywords{Adaptive methods, subgradient, weakly convex, Moreau envelope.}

\section{Introduction}

Stochastic first-order methods are of core practical importance for solving numerous optimization problems including training deep learning networks. Standard stochastic gradient descent (\textsc{Sgd}) \cite{robbins1985stochastic} has become a widely used technique for the latter task. However, its convergence crucially depends on the tuning and update of the learning rate over iterations in order to control the variance of the gradient (due to the use of a small mini-batch for approximating the full gradient) in the stochastic search directions, especially for nonconvex functions. To overcome this issue, several improved variants of \textsc{Sgd} that automatically update the search directions and learning rates using a metric constructed from the history of iterates have been proposed. The pioneering work in this line of research is \cite{jacobs1988increased}, while recent developments include \cite{duchi2011adaptive,zeiler2012adadelta,tielemandivide,kingma2014adam,dozat2016incorporating,reddi2018convergence,nazari2019dadam,nazari2019adaptive,luo2019adaptive}.

Exponential moving average (EMA)-type adaptive methods such as \textsc{RMSprop} \cite{tielemandivide} and \textsc{Adam} \cite{kingma2014adam}, which use exponential moving average of past gradients to update search directions and learning rates simultaneously, can achieve significantly better performance compared to \textsc{Sgd}. However, as shown in \cite{reddi2018convergence}, these techniques may not converge to stationary points when a constant mini-batch size is employed. To address this issue, there is recent work on modified variants of EMA-type algorithms in the nonconvex setting by \cite{zaheer2018adaptive,ward2018adagrad,nazari2019dadam,chen2018convergence,zhou2018convergence}.
However, to the best of our knowledge, all the aforementioned developments require explicitly the objective function to be smooth. Otherwise, it remains unclear whether EMA-type methods still converge or not.

Zeroth-order (gradient-free) optimization has received a lot of interest in diverse optimization and machine learning fields where explicit expressions of the gradients are hard or even infeasible to obtain. Recent examples include, but are not limited to, parameter inference of black-box systems \cite{fu2002optimization,lian2016comprehensive}, black-box classification and also generating adversarial examples from such systems \cite{chen2017zoo,chen2019zo}. Zeroth-order algorithms approximate the gradient with estimates that only utilize values of the objective function \cite{brent2013algorithms,shamir2017optimal,duchi2015optimal,nesterov2017random,shalev2012online}. Although many (stochastic) zeroth-order algorithms have recently been developed and analyzed \cite{balasubramanian2018zeroth, chen2019zo, liu2018signsgd, liu2017zeroth,liu2018zeroth}, they are mainly designed for convex and/or smooth settings, which limits their applicability to a wide range of nonsmooth \& nonconvex machine learning problems.
\begin{table}[t]
 \centering
\begin{minipage}[b]{0.97\linewidth}
  \resizebox{\columnwidth}{!}{
\begin{tabular}{|c|c|c|c|c|c|c|}
  \hline
 {Setting}  & {Algorithm}  & {Adaptive?}  & {Stationary Measure} & $\alpha_t $ &  $\mu$  & {Rate}  \\ \hline
 \addlinespace
\multicolumn{7}{c}{$M_\psi$-smooth loss function}
\\ \hline
 \multirow{2}*{First-order} & RSG  \cite{Ghadimi13:stochastic_nonconvex} & \textsc{No}   & $\EE[\| \nabla \psi(x_{t^*})\|^2]$ & $O(\min\{\frac{1}{M_\psi},\frac{1}{\sqrt{T}}\})$ & -- & $O(\frac{1}{\sqrt{T}})$ \\  \cline{2-7}
 & \textsc{Adam}  \cite{zaheer2018adaptive} & \textsc{Yes}& $\EE [\| \nabla \psi(x_{t^*})\|^2]$ & $\frac{\alpha}{2M_\psi}$ & --&$O(\frac{1}{T} +\frac{1}{b})$ \\  \hline
 \multirow{1}*{Zeroth-order} & RSGF  \cite{Ghadimi13:stochastic_nonconvex} & \textsc{No}  & $\EE[\| \nabla \psi(x_{t^*})\|^2]$& $O(\min\{ \frac{1}{dM_\psi},\frac{1}{\sqrt{dT}}\})$& $O(\frac{1}{d\sqrt{T}})$ & $O(\frac{\sqrt{d}}{\sqrt{T}}+\frac{d}{T})$ \\
 \hline
 \addlinespace
\multicolumn{7}{c}{$M_\psi$-smooth loss function + convex regularizer}
\\ \hline
  \multirow{1}*{First-order} & RSPG  \cite{ghadimi2016mini} &\textsc{No}  & $\EE[\| G_{\mathcal{X}}(x_{t^*})\|^2]$ &  $\frac{\alpha}{2M_\psi}$ & -- &$O(\frac{1}{\sqrt{T}}+ \frac{1}{b})$ \\  \hline
    \multirow{1}*{Zeroth-order}  & RSPGF  \cite{ghadimi2016mini} & \textsc{No}  & $ \EE [\| G_{\mathcal{X}}(x_{t^*})\|^2]$ & $\frac{\alpha}{2M_\psi}$  & $O(\frac{1}{\sqrt{dT}})$ &$O(\frac{d}{b} +\frac{d^2}{bT})$ \\   \hline \addlinespace
\multicolumn{7}{c}{$\rho$-weakly convex loss function+  convex regularizer} \\ \hline
 \multirow{2}*{First-order}  &PSG \cite{davis2019stochastic} & \textsc{No} &$\EE [\|\nabla \psi_{1/(2 \rho)}( x_{t^*})\|^2 ]$  & $O(\frac{1}{\sqrt{T}})$&--& $O(\frac{1}{\sqrt{T}})$  \\
 \cline{2-7}
 &  \textbf{\textsc{Fema} } &   \textsc{Yes}   & \textbf{$\EE [\|\nabla \psi_{1/(2 \rho),Q}(x_{t^*})\|^2 ]$} &  \textbf{  $O(\frac{1}{\sqrt{T}})$} & -- & \textbf{ $O(\frac{1}{\sqrt{T}})$ }  \\ \hline
 Zeroth-order &   \textbf{\textsc{Zema}}&   \textsc{Yes}   &  \textbf{$\EE [\|\nabla \psi_{1/(2 \rho),Q}(x_{t^*})\|^2 ]$} & \textbf{{$O(\frac{1}{\sqrt{T}})$}} & $O(\frac{d}{\sqrt{T}})$ & \textbf{$O(\frac{d^2}{\sqrt{T}})$} \\ \hline
 \end{tabular}
 } 
\end{minipage}
  \caption{Comparison of the representative stochastic first and zeroth-order algorithms for non-convex stochastic problems. Here, $T$ is the total number of iterations; $d$ is the dimension of $x$; $\alpha$ is a constant;  $M_\psi$ is the Lipschitz constant for $\nabla \psi(x)$; $b$ is the batch-size ; $\mu$ is the smoothing parameter; $\alpha_t$ is the stepsize at the $t^{th}$ iteration; $G_{\mathcal{X}}(x_t)=\alpha_t^{-1}\|x_t-x_{t-1}\|$ is the gradient mapping at $x_t$; and $\nabla \psi_{\zeta,Q}(x) $ is the scaled Moreau envelope for a given positive definite matrix  $Q$  (see, Subsection \ref{sec:morea}).
}\label{tab:1}
\end{table}

\paragraph{Contributions.}
This paper focuses on developing EMA-type methods for solving an important class of nonsmooth \& nonconvex  stochastic programming problems. The key technical contributions are:

\begin{itemize}
\item [(i)] \textit{A First-order EMA-type method (\textsc{Fema}).}
Developing a family of EMA-type methods, when access to stochastic subgradient information is assumed, and subsequently establishing that under mild technical conditions, each member of the family converges to first-order stationary points of the stochastic weakly convex objective function. The conditions provided are shown to aid convergence for certain stochastic algorithms that are not based on adaptive optimization \cite{davis2019stochastic,duchi2018stochastic}. In contrast to these results, we first define the scaled Moreau envelop as the stationary measure and subsequently derive the rate of convergence for each member of the proposed EMA-type methods, which depends explicitly on the adaptive moment estimate parameters.

\item [(ii)] \textit{A Zeroth-order EMA-type method (\textsc{Zema}).}
Providing the first sample complexity bounds for a number of popular adaptive zeroth-order algorithms for a reasonably broad class of nonsmooth \& nonconvex optimization problems and a novel adaptive zeroth-order algorithm, called \textsc{Zema}. Theorems~\ref{thm:stochastic_subzero} and \ref{thm:stochastic_sub2zero} establish convergence guarantees of \textsc{Zema} under certain useful parameter settings for solving weakly convex problems.

\item [(iii)] \textit{Guideline for selecting parameters of EMA-type methods.} Our analysis also provides practical suggestions for the parameters of EMA-type methods in the weakly convex setting. In contrast to
\cite{zaheer2018adaptive}, our theoretical analysis of EMA-type methods does not require big-batch stochastic gradients, and further establishes the convergence rate of first and zeroth order stochastic EMA-type methods to a broader class of objective functions.
\end{itemize}

\paragraph{Outline.}
In the next section, we introduce necessary definitions and present preliminary results on subdifferentials and the Moreau envelope. Section~\ref{sec1} provides a detailed description of the proposed adaptive methods and establishes theoretical guarantees for such methods for solving weakly convex problems. Section~\ref{Experiments} presents numerical experiments to assess the performance of the proposed algorithms, while some concluding remarks are drawn in Section~\ref{conclusion}.


\section{Mathematical preliminaries and notations}

\subsection{Notation}Throughout, $\Bbb{R}_+$ and $\Bbb{R}^d$ denote the sets of nonnegative real numbers and real coordinate space of $d$ dimensions, respectively. We let $\mathcal{S}^d_{++} $ denote the set of all symmetric positive definite $d\times d$ matrices. The minimum and maximum eigenvalues of the matrix $Q$ are denoted by $\lambda_{\min}(Q)$ and $\lambda_{\max}(Q)$, respectively. For any vectors $a,b\in\Bbb{R}^d$, we use $\sqrt{a}$ to denote element-wise square root, $a^2$ to denote element-wise square, $a/b$ to denote element-wise division, $\max(a,b)$
to denote element-wise maximum and $ \langle a , b \rangle$ to denote the standard Euclidean inner product. For any positive integers $d$ and $T$, we set $[d]:=\{1,\ldots,d\}$ and $(T]:=\{0,\ldots,T\}$. Further, for any vector $x_t\in \Bbb{R}^d$, $(x_{t})_i$ denotes its $i^{th}$ coordinate where $i\in[d]$.
We use $\|\cdot\|$, $\|\cdot\|_1$ and $\|\cdot\|_{\infty}$ to denote the $\ell_2$-norm, $\ell_1$-norm and the infinity norm, respectively. We let $\text{diag}(x)$ denote the diagonal matrix with diagonal entries $x_1, \ldots , x_d$. We also let $\vec{\mathbf{1}}$ denote all-ones vector.

For any matrix $Q\in \mathcal{S}^d_{++}$ and any set $\mathcal{X}$, $\Pi_{\mathcal{X},Q}~[x]$ denotes the scaled Euclidean projection of a vector $x$ onto $\mathcal{X}$, i.e.,
\begin{equation}\label{proj}
\Pi_{\mathcal{X},Q}~\big[x\big]=\arg\min_{y\in \mathcal{X}} \|Q^{1/2}(x-y)\|^2.
\end{equation}
Given two sequences $\{a_n\}$ and $\{b_n\},$ we write $a_n=O(b_n)$ if there exists a constant $0<c<+\infty$ such that $a_n\leq c b_n$. The expectation conditioned on all the realizations $\xi_0,\xi_1,\ldots, \xi_{t-1} \sim P$ is denoted by $\mathbb{E}_{t}[\cdot]$.

\subsection{Subdifferentials, weak convexity and scaled Moreau envelope} \label{sec:morea}

In this section, we will present some preliminaries, describing the setup and reviewing some background material that will be useful in the sequel.

\begin{defn}[domain and epigraph]\textnormal{\cite{rockafellar1970convex}} For a function $\psi:\RR^d \rightarrow \RR \cup \{+\infty\}$, let
\begin{equation}\label{eqn:epi}
  \dom \psi:=\{x\in \Bbb{R}^d:\psi(x)<+\infty\},\quad \textnormal{and} \quad
  \epi \psi := \{(x, z)\in \RR^d \times \RR \mid \psi(x) \leq z\}.
\end{equation}
Then, a function $\psi$ is said to be \emph{closed} if $\epi \psi$ is a closed set. A function $\psi$ is called \emph{proper} if $\dom~\psi \neq \emptyset$, i.e., a function is proper if its effective domain is nonempty and it never attains $+\infty$.
\end{defn}

\begin{defn} [Lipschitz continuity and smoothness] A function $\psi$ is $L_\psi$-Lipschitz continuous if, for some $L_\psi \in \R_{+}$,
$$
|\psi(x) - \psi(y) | \leq L_\psi \|x-y\|, ~~ \forall  ~~x,y \in \dom \psi,
$$
Further, a function $\phi$ is $M_\phi$-smooth if its gradient $\nabla \phi $ is $M_\phi$-Lipschitz continuous.
\end{defn}

\begin{defn}[subdifferential]\textnormal{\cite{rockafellar1970convex}}
For a function $\psi \colon \RR^d \rightarrow \RR\cup \{+\infty\} $ and a point $x \in \dom~ \psi$, we let
\begin{align*}  \label{eqn:weak_subdiff}
\partial \psi(x) =  \big\{ w \in \RR^d \mid  \; \psi(y) \geq \psi(x) + \dotp{w, y-x} + o(\| Q^{1/2}(y - x)\|), \quad \textnormal{as} \quad y\rightarrow x \big\}
\end{align*}
denote the \textit{Fr\'{e}chet subdifferential} of $\psi$ at $x$. We set $\partial \psi(x) = \emptyset$ if $x \notin \dom~\psi$.
\end{defn}


Next, we introduce the the concept of $(\rho,Q)$-weakly convex and discuss some important properties of this class of nonconvex (possibly nonsmooth) functions.

\begin{defn} [convexity and weak convexity] \label{def:weak}
A function $\psi \colon \RR^d \rightarrow \RR\cup \{+\infty\} $ is called convex if, for all $x,y \in  \RR^d $ and $ t \in [0, 1]$,
$$
\psi\big(tx+(1-t)y\big) \leq t\psi(x) + (1-t)\psi(y).$$
A function $\psi$ is said to be $(\rho,Q)$-weakly convex if, for some $\rho\in \R $ and some $Q\in \mathcal{S}^d_{++}$, the function $x\mapsto \psi(x)+\frac{\rho}{2}\|Q^{1/2}x\|^2$ is convex.
\end{defn}
In the case when $Q=I$, $\psi(x)$ is called $\rho$-weakly convex \cite{vial1983strong}. It is worth mentioning that weak convexity is a special yet very common case of nonconvex functions, which contains all convex (possibly nonsmooth) functions and Lipschitz smooth functions. For example, if $\psi$ is $M_\psi$-smooth, then it is $M_\psi$-weakly convex. We refer the reader to \cite{davis2019stochastic,duchi2018stochastic,davis2019proximally,zhang2018convergence} for some examples and applications.

%
%
The following lemma provides some useful properties of the $(\rho,Q)$-weakly convex functions. The proof is given in the appendix.

\begin{lem}[subdifferential characterization]\label{lem:weak:hyp}
Let $\psi:\Bbb{R}^d\rightarrow \Bbb{R}\cup \{+\infty\}$ be a closed function. Then, the following are equivalent:
\begin{enumerate}
  \item [\textnormal{(i)}] $\psi$ is $(\rho,Q)$-weakly convex for some $\rho\in \R$ and some $Q\in \mathcal{S}^d_{++}$.
  \item [\textnormal{(ii)}] For all $ x,y\in\Bbb{R}^d$ with $\varpi\in \partial \psi(x),$
  \begin{equation}\label{eqn:stronger_ineq}
    \psi(y)\geq \psi(x)+\langle \varpi,y-x \rangle-\frac{\rho}{2}\|Q^{1/2}
    (y-x) \|^2.
  \end{equation}
  \item [\textnormal{(iii)}] The subdifferential map is hypomonotone. That is, for all $x,y\in \Bbb{R}^d$ with  $\omega \in \partial \psi(x)$ and $\varpi\in \partial \psi(y)$,
  \begin{equation}\label{eqn:stronger_ineq_hypo}
    \langle \omega-\varpi,x-y \rangle\geq -\rho \|Q^{1/2}(y-x)\|^2.
  \end{equation}
\end{enumerate}
\end{lem}

Next, we give the definition and some important properties of the scaled Moreau envelope function and the associated scaled proximal mapping.

\begin{defn}[Moreau envelope and proximal mapping]
 For a proper, lower semicontinuous function $\psi:\RR^d \rightarrow \RR \cup \{+\infty\}$, a positive constant $\zeta$ and $Q\in \mathcal{S}^d_{++}$,  the scaled Moreau envelope and the scaled proximal mapping are defined respectively as
\begin{equation}\label{def:moreau}
 \psi_{\zeta,Q}(x):=\min_{y }~ \left\{\psi(y)+\tfrac{1}{2\zeta}\|Q^{1/2}(x-y)\|^2\right\} ,
\end{equation}
	\begin{equation}\label{prox}
	\textnormal{prox}_{\zeta \psi,Q}(x):=\arg\min_{y}\, \left\{\psi(y)+\tfrac{1}{2\zeta}\|Q^{1/2}(x-y)\|^2\right\}.
	\end{equation}
In the case when $Q=I$, \eqref{def:moreau} and \eqref{prox} reduce to the (unscaled) Moreau envelope and proximal mapping which are denoted by $ \psi_{\zeta}(x) $  and $\textnormal{prox}_{\zeta \psi}(x)$, respectively.
\end{defn}

The Moreau envelope and the proximal mapping play crucial roles in numerical analysis and optimization both theoretically and computationally \cite{attouch1984variational,rockafellar1970convex,parikh2014proximal}. It can be easily shown that the set of minimizers of $  \psi_{\zeta,Q}(x)$ coincides with the set of minimizers of $ \psi(x)$. The scaled proximal mappings share some key properties with the conventional ones \cite{lee2014proximal}. We collect a few of them in the following lemma.
\begin{lem} [properties of proximal mapping] \label{mc}
Let $\psi:\Bbb{R}^d\rightarrow \Bbb{R}\cup \{+\infty\}$ be a $(\rho,Q)$-weakly convex function. Then, for all $\zeta\in (0,\rho^{-1})$, and  $x\in \Bbb{R}^d$ the following hold:
\begin{itemize}
\item [\textnormal{(i)}] $\textnormal{prox}_{\zeta \psi,Q}(x)$ exists and is unique.
\item [\textnormal{(ii)}] $\textnormal{prox}_{\zeta \psi,Q}(x)$ is nonexpansive. That is, for all $y\in \Bbb{R}^d$,
  \begin{equation*}
  \|Q^{1/2}\big(\textnormal{prox}_{\zeta \psi,Q}(x)-\textnormal{prox}_{\zeta \psi,Q}(y)\big)\|^2 \le \|Q^{1/2}(x-y)\|^2.
  \end{equation*}
\end{itemize}
\begin{proof}
Since the function $\psi(y)+\tfrac{1}{2\zeta}\|Q^{1/2}(x-y)\|^2$ is strongly convex for any $\zeta\in (0,\rho^{-1})$, (i) follows.  To show (ii), we write $u=\textnormal{prox}_{\zeta \psi,Q}(x)$ and $ v= \textnormal{prox}_{\zeta \psi,Q}(x)$. Then,
$$
  \langle (u-v), Q (x-y) \rangle \geq \|Q^{1/2}(u-v)\|^2.
$$
The above inequality together with Cauchy-Schwarz inequality \footnote{ $\forall a,b: ~~ 2 \dotp{a,b} \leq \|a\|^2 +\|b\|^2 $} implies (ii).
\end{proof}
\end{lem}

The following lemma extends the result of \cite{rockafellar1970convex} to the $(\rho, Q)$-weakly convex functions and shows that the scaled Moreau envelope is differentiable, with a meaningful gradient. The proof is almost identical to that of \cite[Theorem 31.5]{rockafellar1970convex}.


\begin{lem}[smoothness of the envelope]\label{lem:moreau_grad} Let $\psi\colon\R^d\to \R\cup\{+\infty\}$ be a $(\rho,Q)$-weakly convex function. Then, for any  $\zeta\in (0,\rho^{-1})$, the scaled Moreau envelope $\psi_{\zeta,Q}$ is continuously differentiable at any point $x \in \dom~\psi$ with gradient given by
\begin{equation}\label{eqn:grad_form}
  \nabla \psi_{\zeta,Q}(x)=\zeta^{-1} Q(x- \textnormal{prox}_{\zeta \psi,Q}(x)).
\end{equation}
\end{lem}

Our convergence analysis of EMA-type methods leverages the scaled Moreau envelope's connections to \textit{scaled near stationarity}:\begin{eqnarray}\label{grad}
\nonumber
&& \|Q(\bar{x} -x)\| =\zeta\| \nabla \psi_{\zeta,Q}(x)\|,\\
\nonumber
&& \psi(\bar{x}) \leq \psi(x), \\
&& \text{dist}\big(0;\partial \psi(\bar{x})\big) \leq  \| \nabla \psi_{\zeta,Q}(x)\|,
  \end{eqnarray}
 where $\bar{x}=\textnormal{prox}_{\zeta \psi,Q}(x)$.

In the case when $Q=I$, these conditions reduced to the ones derived in \cite{davis2019stochastic}. The above conditions imply that any nearly stationary point $x$ of $\psi_{\zeta,Q}$ is close to a nearly stationary point $x$ of the original function $\psi(\cdot)$. Thus, in order to prove convergence of our algoithms to the stationary points, it is sufficient to show that $\nabla \psi_{\zeta,Q}(x)\rightarrow 0$. For further details and a historical account, we refer the interested reader to \cite{davis2019stochastic}.

\subsection{Review of adaptive gradient methods}\label{sec:revi}

Our algorithm is based on the EMA-type adaptive gradient algorithms
\cite{tielemandivide,kingma2014adam,reddi2018convergence} designed for minimization of a function $\psi(x)$ subject to the constraint $x \in \mathcal{X}$. We begin by describing the standard version of these algorithms, and then discuss the extensions for the weakly convex setting of interest in this paper.

Adaptive gradient algorithms generate a sequence $\{x_t\}_{t=0}^T$ of iterates of the form
\begin{equation*}
 x_{t+1} = \Pi_{\mathcal{X},\widehat{V}_t^{1/2}} \big[x_{t}- \alpha_t \widehat{V}_t^{-1/2} m_{t} \big],
\end{equation*}
where $\widehat{V}_t=\text{diag}(\widehat{\upsilon}_{t})$; $\mathcal{X}$ is a closed convex set; $\alpha_t$ is the stepsize; $m_t$ and $ \widehat{\upsilon}_{t}$ are \textit{averaging} functions relate to the historical gradients. Throughout, we refer $\alpha_t \widehat{V}_t^{-1/2}$ as learning rate. We note that the standard (stochastic) gradient descent falls within this framework by setting:
$$
m_t =g_t, \quad \text{and} \quad
\widehat{\upsilon}_{t}= \vec{\mathbf{1}}.
$$

One practical issue for the stochastic gradient descent is that in order to ensure convergence of it, the $\alpha_t$ has to decay to zero which leads to slower convergence. The key idea of adaptive methods is to choose averaging functions appropriately, so as to ensure good convergence. \textsc{Adagrad} \cite{duchi2011adaptive} uses
\begin{eqnarray*}
 m_t=g_t, \quad \text{and} \quad  \widehat{\upsilon}_{t}  = \frac{1}{t}\sum_{s=1}^{t}g_s^2,
\end{eqnarray*}
and stepsize $\alpha_t=\alpha/\sqrt{t}$ for all $t\in(T].$ This setting for function $m_t$ gives a modest stepsize decay of $\alpha/\sqrt{\sum_{s}g_{s,j}^2}$ for $j\in [d]$. Further, when the gradients are sparse, this setting gives significant improvement in the empirical performance of gradient-based methods. Although \textsc{Adagrad} works well in sparse convex settings, its performance appears to deteriorate in (dense) nonconvex settings. To tackle this issue, variants of \textsc{Adagrad} such as \textsc{Adam} and \textsc{RMSProp} have been proposed, which replace the sum of the outer products with an exponential moving average. In particular, \textsc{Adam} \cite{kingma2014adam} uses the following recursion formulas:
\begin{eqnarray*}
\nonumber
  m_{t} &=&\beta_{1,t}m_{t-1}+(1-\beta_{1,t})g_{t}, \\
 \widehat{\upsilon}_{t}  &=& \beta_2  \widehat{\upsilon}_{t-1}+(1-\beta_2)g^2_{t},
\end{eqnarray*}
where $\{\beta_{1,t}\}_{t=0}^T,\beta_2 \in [0,1)$ and $m_{-1} =\widehat{\upsilon}_{-1} = 0$. In the case when $\beta_{1,t} = 0$, \textsc{Adam} reduces  to \textsc{RMSProp} \cite{tielemandivide}. In practice the momentum term arising due to non-zero $\beta_{1,t}$ appears to significantly improves the performance.

More recently, the work \cite{reddi2018convergence} pointed out the convergence issues of \textsc{Adam} even in the convex setting and proposed \textsc{AMSGrad}, a corrected version of \textsc{Adam}. \textsc{AMSGrad} uses the following updates:
\begin{eqnarray*}
\nonumber
m_{t} &=&\beta_{1,t}m_{t-1}+(1-\beta_{1,t})g_{t}, \\
\nonumber
\upsilon_{t} &=& \beta_2\upsilon_{t-1}+(1-\beta_2)g^2_{t}, \\
 \widehat{\upsilon}_{t} &=& \max(\widehat{\upsilon}_{t-1} ,  \upsilon_{t}),
\end{eqnarray*}
where  $\{\beta_{1,t}\}_{t=0}^T,\beta_2 \in [0,1)$ and $m_{-1} =\upsilon_{-1}=\widehat{\upsilon}_{-1} = 0$.

We will mainly focus on \textsc{AMSGrad} algorithm due to this generality but our arguments also apply to \textsc{Adam} and other algorithms such as \textsc{Adagrad}.

\section{Algorithms and convergence analysis}\label{sec1}

In this section, we introduce a new class of EMA-type algorithms for solving composite problems of the form
\begin{equation} \label{eqn:problem_class}
\min_{x \in \R^d} \psi(x)= f(x) + h(x),
\end{equation}
where $\psi, f$ and $h$ satisfy the following assumption.
\begin{assum}\label{it1}
$f$ is $(\rho,Q)$-weakly convex for some $\rho \in \R$ and $Q= \textnormal{diag}(q) $ where $q \in \R^d_{++}$ ; $h$ is closed convex; and $\psi$ is bounded below over its domain, i.e., $\psi^*=\min_{x \in \R^d} \: \psi(x)$ is finite.
\end{assum}
We also assume that we only have access to noisy subgradients about the function $f(x)$. In particular, in the first-order setting, we assume that problem \eqref{eqn:problem_class} is to be solved by EMA-type adaptive methods which acquire the subgradients of $f$ via subsequent calls to a stochastic first-order oracle (\textit{SFO}) that satisfies the following assumption.
\begin{assum}\label{it2}
Let $(\Omega, \mathcal{F}, P)$ denote a probability space and equip $ \Bbb{R}^d$ with the Borel $\sigma$-algebra. Then, it is possible to generate i.i.d.\ realizations $\xi_0,\xi_1, \ldots \sim P$. Further, there is an open set $U$ containing $\dom h$ and a measurable mapping $G \colon U \times \Omega \rightarrow \RR^d$ satisfying  $$\EE_{\xi}[G(x,\xi)]\in \partial f(x), \quad  \textnormal{for all} \quad  x\in U. $$
\end{assum}
Assumption~\ref{it2} is standard in analysis of  stochastic subgradient methods \cite{davis2019stochastic,duchi2018stochastic} and it is identical to Assumptions (A1) and (A2) in \cite{davis2019stochastic} .

\subsection{First-order EMA-type method (\textsc{Fema})}
\label{sec:fema}

\begin{algorithm}[t]
\caption{First-order EMA-type method (\textsc{Fema})}\label{alg:subgradient}
\SetKwInput{Input}{input~}
\SetKwInput{Output}{output~}
\SetKwInput{optionone}{option~I~}
\SetKwInput{optiontwo}{option~II~}
\Input{ Initial point $x_0 \in \dom h$, number of iterations $T$,  stepsizes $\{\alpha_t\}_{t=0}^T\subset\R_+$, decay parameters $\{\beta_{1,t}\}_{t=0}^T, \beta_2, \beta_3 \in [0,1)$ and a vector $q \in \mathbb{R}_{+}^d$ satisfying Assumption~\ref{it1}\;}
\BlankLine
Initialize $m_{-1}= \upsilon_{-1}=0 $ and $\widehat{\upsilon}_{-1}=q$\;
\For{ $0\leftarrow t : T$ }{\label{forins}
Draw sample $\xi_{t}$ from $P$ \;
$g_t =  G(x_t,\xi_{t})$\;
$m_{t}= \beta_{1,t} m_{t-1}+(1-\beta_{1,t})g_{t}$\; 
$\upsilon_{t}= \beta_2\upsilon_{t-1}+(1-\beta_2) g_{t}^2$\;
 $\widehat{\upsilon}_{t}= \beta_{3} \widehat{\upsilon}_{t-1}+(1-\beta_{3}) \max (\widehat{\upsilon}_{t-1},\upsilon_t)$, and  $\widehat{V}_t =  \text{diag}(\widehat{\upsilon}_{t})$\;
 $ x_{t+1}=\textnormal{prox}_{\alpha_th, \widehat{V}_t^{1/2}}\big(x_{t} - \alpha_t\widehat{V}_t^{-1/2} m_t\big)$\label{m11} \;
}

\BlankLine
\Output{Choose $x_{t^*}$ from $\{x_t\}_{t=0}^{T}$ with probability $\mathbb{P}(t^*=t)=\alpha_t/\sum_{t=0}^T \alpha_t$.}
\end{algorithm}

Our approach, as given in Algorithm \ref{alg:subgradient}, is a new model-based adaptive method using EMA of past gradients and is inspired by \textsc{AMSGrad}. Similar to this method, \textsc{Fema} scales down the gradient by the square roots of EMA of past squared gradients. However, \textsc{Fema} takes a larger step toward the optimum in comparison to \textsc{AMSGrad} ($\beta_3=0$) and yet incorporates the intuition of slowly decaying the effect of previous gradients on the learning rate.

Next, a few remarks about the \textsc{Fema} method are in order: i) The vector operations such as square root and the maximum operators are taken elementwise;  ii) At iteration $t$, with input $x_t$, the $\emph{SFO}$ outputs a stochastic subgradient $G(x_t, \xi_{t})$, where $\xi_{t}$ and $G(x_t, \xi_{t})$ satisfy Assumption~\ref{it2} for all $t \in (T]$; iii) \textsc{Fema} includes well-known adaptive algorithms as special cases, including \textsc{Adagrad}, \textsc{Adam}, \textsc{RMSprop}, \textsc{AMSGrad}, and \textsc{Sgd} (see, Subsection \ref{sec:revi}).

We establish convergence analysis of \textsc{Fema}  under Assumptions~ \ref{it1} and \ref{it2} for solving stochastic composite problem \eqref{eqn:problem_class}. We first present a simplified argument for the case of $h$ being the indicator function of a nonempty closed convex set $\cX$. In particular, we provide convergence guarantees of Algorithm~\ref{alg:subgradient} for solving the constrained $(\rho,Q)$-weakly convex problems of the following form
\begin{equation}\label{eqn:target_constr}
\min_{x\in\cX}~  \psi(x)= f(x),
\end{equation}
where $\mathcal{X}$ is a nonempty closed convex set.

\begin{thm}[\textbf{projected \textsc{Fema}}]\label{Sthm:stochastic_sub}
Suppose Assumptions~\ref{it1}--\ref{it2} hold, $\|G(x,\cdot)\|_{\infty} \leq G_{\infty}$ and $\|x-y\|_{\infty} \leq D_{\infty}$ for all $x, y\in \cX$. Further, for all $ t \in (T]$, let
\begin{align*}
0 \leq \beta_{1,t}\leq\beta_1,~~~\{\beta_i\}_{i=1}^3  \in [0,1), ~~~\tau = \frac{\beta_1}{\sqrt{\beta_2}}<1, \quad \textnormal{and} \quad  \bar{\rho} >\rho.
\end{align*}
Then, for $x_{t^*}$ generated by Algorithm~\ref{alg:subgradient}, we have
\begin{eqnarray}\label{Projected}
&&\EE \left[\|\nabla \psi_{1/\bar \rho,\widehat{V}_{t^*}^{1/2}}(x_{t^*})\|^2\right] \leq
 \frac{ \bar \rho \Delta_{\psi}+ \bar \rho^2  (\sum_{t=0}^{T} \alpha_t^2 C_{1,t}+  C_{2,T} )}{(\bar \rho-\rho)(1-\beta_1)\sum_{t=0}^T \alpha_t },
\end{eqnarray}
where $\Delta_{\psi} = \psi_{1/\bar \rho,\widehat{V}_{-1}^{1/2}}(x_{0}) - \psi^*$,
\begin{align}\label{pj}
\nonumber
C_{1,t} &= \frac{  \sum_{k=0}^t\tau^{t-k} \EE [\|g_{k}\|_1]}{(1-\beta_1)\sqrt{(1-\beta_2)(1-\beta_3)}} , \quad \textnormal{and}\\
C_{2,T} &= \frac{D_{\infty}^2}{2} \big(\sum_{t=0}^T \beta_{1,t}^2 \EE[ \|\widehat{\upsilon}_{t}^{1/2} \|_1] +\EE[ \|\widehat{\upsilon}_{T}^{1/2} \|_1]\big).
\end{align}
\end{thm}
\begin{proof}
Let $\bar x_t =\textnormal{prox}_{\psi/\bar \rho,\widehat{V}_t^{1/2}}(x_t)$. For the sequence $\{x_t\}_{t=0}^T$ generated by Algorithm~\ref{alg:subgradient}, we have
\begin{eqnarray}\label{myeqn:keydef}
\nonumber
   &&  \EE_t[ \psi_{1/\bar \rho,\widehat{V}_t^{1/2}}(x_{t+1})] \leq \EE_t[ f(\bar x_t) + \frac{\bar \rho}{2} \|\widehat{V}_t^{1/4}(x_{t+1} - \bar x_t)\|^2]\notag\\
&=& f(\bar x_t) + \frac{\bar \rho}{2}\EE_t[ \|\widehat{V}_t^{1/4}\big(\Pi_{\cX,\widehat{V}_t^{1/2}}(x_{t} - \alpha_t \widehat{V}_t^{-1/2}m_t) -  \Pi_{\cX,\widehat{V}_t^{1/2}}(\bar x_t) \big)\|^2] \nonumber\\
&\leq & f(\bar x_t) + \frac{\bar \rho}{2}\EE_t[ \|\widehat{V}_t^{1/4}\big(x_{t}  - \bar x_t- \alpha_t \widehat{V}_t^{-1/2}m_t\big)\|^2] \nonumber\\
\nonumber
&\leq & f(\bar x_t) + \frac{\bar \rho}{2} \EE_t[\|\widehat{V}_t^{1/4}(x_{t} - \bar x_t)\|^2] \\
&+& \bar \rho\alpha_t\EE_t[ \dotp{\bar x_t-x_t , m_t}] + \frac{\alpha_t^2\bar \rho}{2}\EE_t[ \|\widehat{V}_t^{-1/4}m_t\|^2],
\end{eqnarray}
where the first inequality follows from the definition of the scaled proximal map in \eqref{prox} and the second inequality uses Lemma~\ref{mc}(ii).

Now, for the third term on the right hand side of \eqref{myeqn:keydef}, it follows from the update rule of $m_t$ in Algorithm~\ref{alg:subgradient} that
\begin{subequations}
\begin{equation} \label{eq:thm012}
\EE_t[ \dotp{\bar x_t-x_t , m_t}] = \beta_{1,t}\dotp{\bar x_t-x_t , m_{t-1}}  + (1-\beta_{1,t})\EE_t[ \dotp{\bar x_t-x_t , g_t}].
\end{equation}
On the right hand side of this equality, we multiple and divide the first term by $ \alpha_t^{1/2} \widehat{V}_{t}^{-1/4}$ and then use the Cauchy-Schwarz inequality to obtain
\begin{eqnarray}\label{eq:thm12}
\nonumber
&& \dotp{ m_{t-1}, \beta_{1,t} (\bar x_t-x_t) }\\
\nonumber
&\leq & \frac{\alpha_t}{2}\EE_t[\|\widehat{V}_{t}^{-1/4}  m_{t-1}\|^2]+ \frac{\beta_{1,t}^2}{2\alpha_t} \EE_t[\|\widehat{V}_{t}^{1/4}(\bar{x}_t-x_t)\|^2]\\
&\leq & \frac{\alpha_t}{2}\|\widehat{V}_{t-1}^{-1/4}  m_{t-1}\|^2+ \frac{\beta_{1,t}^2}{2\alpha_t} \EE_t[\|\widehat{V}_{t}^{1/4}(\bar{x}_t-x_t)\|^2],
\end{eqnarray}
where the second inequality follows since $(\widehat{V}_{t})_{ii}\geq (\widehat{V}_{t-1})_{ii}$ for all $i \in [d]$.

Observe that
\begin{eqnarray}\label{eq:thm0012}
\nonumber
&& \EE_t[\|\widehat{V}_t^{1/4}(x_{t} - \bar x_t)\|^2] \\
\nonumber
&=& \|\widehat{V}_{t-1}^{1/4}(x_{t} - \bar x_t)\|^2+\EE_t[\|\widehat{V}_t^{1/4}(x_{t} - \bar x_t)\|^2] -\|\widehat{V}_{t-1}^{1/4}(x_{t} - \bar x_t)\|^2\\
\nonumber
&=&
\|\widehat{V}_{t-1}^{1/4}(x_{t} - \bar x_t)\|^2+\EE_t[\|\widehat{V}_t^{1/4}(x_{t} - \bar x_t)\|^2 -\|\widehat{V}_{t-1}^{1/4}(x_{t} - \bar x_t)\|^2] \\
&\leq &
\|\widehat{V}_{t-1}^{1/4}(x_{t} - \bar x_t)\|^2+\EE_t[\|(\widehat{V}_t^{1/4}-\widehat{V}_{t-1}^{1/4})(x_{t} - \bar x_t)\|^2].
\end{eqnarray}
\end{subequations}

Insert \eqref{eq:thm012}, \eqref{eq:thm12} and \eqref{eq:thm0012} in \eqref{myeqn:keydef} to get
\begin{eqnarray*}
\nonumber
&&  \EE_t[ \psi_{1/\bar \rho,\widehat{V}_t^{1/2}}(x_{t+1})]
 -\psi_{1/\bar \rho,\widehat{V}_{t-1}^{1/2}}(x_{t})
 \leq
   \bar \rho(1-\beta_{1,t}) \alpha_t \EE_t[ \dotp{\bar x_t-x_t ,g_t}]\\
   \nonumber
   &+& \frac{ \bar \rho}{2} \EE_t[ \|(\widehat{V}_{t}^{1/4}-\widehat{V}_{t-1}^{1/4})(x_{t} - \bar x_t)\|^2]
  + \frac{\beta_{1,t}^2 \bar \rho}{2} \EE_t[ \|\widehat{V}_{t}^{1/4}(x_{t} - \bar x_t)\|^2]\\
& +&  \frac{\bar \rho  \alpha_t^2 }{2} \big( \EE_t[ \|\widehat{V}_t^{-1/4}m_t\|^2] + \|\widehat{V}_{t-1}^{-1/4} m_{t-1}\|^2 \big).
\end{eqnarray*}
Let $\gamma_t = \EE_t[g_t]$ for all $t \in (T]$. Taking expectations on both sides of the above inequality w.r.t. $\xi_0,\xi_1,\ldots,\xi_{t-1}$ and using the law
of total expectation, we get
\begin{eqnarray}\label{eqn:total:low}
\nonumber
 && \EE[ \psi_{1/\bar \rho,\widehat{V}_T^{1/2}}(x_{T+1})] - \psi_{1/\bar \rho,\widehat{V}_{-1}^{1/2}}(x_{0})
  \leq \bar \rho \sum_{t=0}^T(1-\beta_{1,t}) \alpha_t \EE [ \dotp{\bar x_t-x_t ,\gamma_t}]\\
  \nonumber
  &+ & \frac{ \bar \rho}{2} \sum_{t=0}^T \EE[ \|(\widehat{V}_{t}^{1/4}-\widehat{V}_{t-1}^{1/4})(x_{t} - \bar x_t)\|^2]
 +\beta_{1,t}^2 \EE[ \|\widehat{V}_{t}^{1/4}(x_{t} - \bar x_t)\|^2] \\
&+&  \frac{\bar \rho}{2}  \sum_{t=0}^T \alpha_t^2 \big( \EE[ \|\widehat{V}_t^{-1/4}m_t\|^2] +  \EE[ \|\widehat{V}_{t-1}^{-1/4} m_{t-1}\|^2 ] \big).
\end{eqnarray}	
Next, we upper bound each sum on the R.H.S. of the the above inequality. To bound the last sum, it follows from Lemma~\ref{lem:spar} that
\begin{subequations}
\begin{eqnarray}\label{eq:fcccthmi1}
\nonumber
 && \frac{\bar{\rho}}{2}  \sum_{t=0}^{T}  \alpha_t^2   \big( \EE[ \|\widehat{V}_t^{-1/4}m_t\|^2] +  \EE[ \|\widehat{V}_{t-1}^{-1/4} m_{t-1}\|^2 ] \big) \\
& \leq &  \bar{\rho}  \sum_{t=0}^{T}  \frac{   \alpha_t^2  \sum_{k=0}^t{\tau}^{t-k} \EE[ \|{g_{k}}\|_1] }{{(1-\beta_1)\sqrt{(1-\beta_2)(1-\beta_3)}}} = \bar{\rho}   \sum_{t=0}^{T}  \alpha_t^2 C_{1,t},
\end{eqnarray}
where $ C_{1,t}$ is defined in \eqref{pj}.

To bound the second sum on the R.H.S. of \eqref{eqn:total:low}, we have that
\begin{eqnarray}\label{pp}
\nonumber
 &&  \frac{\bar \rho}{2} \sum_{t=0}^T \EE[ \|(\widehat{V}_{t}^{1/4}-\widehat{V}_{t-1}^{1/4})(x_{t} - \bar x_t)\|^2] + \beta_{1,t}^2 \EE[ \|\widehat{V}_{t}^{1/4}(x_{t} - \bar x_t)\|^2] \\
 \nonumber
  &\leq& \frac{\bar \rho }{2}  \sum_{t=0}^T  \big(\EE [\|\widehat{\upsilon}_{t}^{1/4}-\widehat{\upsilon}_{t-1}^{1/4} \|_1^2]  + \beta_{1,t}^2 \EE [\|\widehat{\upsilon}_{t}^{1/2}\|_1]\big) \|x_t - \bar{x}_t\|_{\infty}^2 \\
 \nonumber
 &\leq&  \frac{\bar \rho D_{\infty}^2}{2}  \sum_{t=0}^T  \EE [\sum_{i=1}^{d}  ((\widehat{\upsilon}_{t})_i^{1/4}-(\widehat{\upsilon}_{t-1})_i^{1/4})^2]  +  \beta_{1,t}^2 \EE [\|\widehat{\upsilon}_{t}^{1/4}\|_1] \\
 \nonumber
 &\leq &  \frac{\bar \rho D_{\infty}^2}{2} \sum_{t=0}^T \EE [\sum_{i=1}^{d} (\widehat{\upsilon}_{t})_i^{1/4}((\widehat{\upsilon}_{t})_i^{1/4}-(\widehat{\upsilon}_{t-1})_i^{1/4})]+ \beta_{1,t}^2 \EE [\|\widehat{\upsilon}_{t}^{1/2}\|_1]\\
&=&\frac{\bar \rho D_{\infty}^2}{2} \big(\EE[\| \widehat{\upsilon}_{T}^{1/2}\|_1]+ \sum_{t=0}^T \beta_{1,t}^2 \EE [\|\widehat{\upsilon}_{t}^{1/2}\|_1]\big) =\bar{\rho} C_{2,T},
\end{eqnarray}
where the first inequality follows from H\"{o}lder's inequality;  the second inequality uses the bounded domain assumption; the third inequality is obtained from the fact that $(\widehat{V}_{t})_{ii}\geq (\widehat{V}_{t-1})_{ii}$ for all $i \in [d]$; the first equality follows from the method of differences for telescoping series; and the last equality uses \eqref{pj}.

To bound the first sum on the R.H.S. of \eqref{eqn:total:low}, using Assumption~\ref{it1} and Lemma~\ref{lem:weak:hyp}(ii), we have
\begin{eqnarray}\label{fklm}
\nonumber
 && \sum_{t=0}^T (1-\beta_{1,t})\alpha_t \EE  [\dotp{\bar x_t- x_t  , \gamma_t }] \\
&\leq &  \sum_{t=0}^T (\beta_{1,t}-1)\alpha_t \EE [f(x_t) - f(\bar x_t) - \frac{\rho}{2}\|\widehat{V}_t^{1/4}(x_t - \bar x_t)\|^2].
\end{eqnarray}
\end{subequations}
By Assumption~\ref{it1}, the function $\psi$ defined in \eqref{eqn:problem_class} is bounded below over its domain, i.e., $\psi^*=\min_{x} \: \psi(x)$ is finite.  Now, substituting \eqref{eq:fcccthmi1}--\eqref{fklm}  into \eqref{eqn:total:low} and utilizing the fact that $\psi^*\leq \psi(x_{T+1})$, we get
\begin{eqnarray}\label{myeqn:main_ineqs}
&&\frac{1}{\sum_{t=0}^T \alpha_t}  \sum_{t=0}^T (1-\beta_{1,t})\alpha_t \EE [f(x_t) - f(\bar x_t) - \frac{\rho}{2}\|\widehat{V}_t^{1/4}(x_t - \bar x_t)\|^2]\notag\\
 &\leq&  \frac{1}{\bar \rho\sum_{t=0}^T \alpha_t}
\Big(\psi_{1/\bar \rho,\widehat{V}_{-1}^{1/2}}(x_{0}) - \psi^* +\bar{\rho} \sum_{t=0}^{T} \alpha_t^2 C_{1,t}
+ \bar{\rho} C_{2,T} \Big).
\end{eqnarray}
Further, by Assumption~\ref{it1}, the function $x\mapsto f(x)+\frac{\bar \rho}{2}\|\widehat{V}_{t^*}^{1/4}(x-x_{t^*})\|^2$ is strongly convex with the modulus $\bar \rho-\rho >0$. Hence, for $x_{t^*}$ generated by  Algorithm~\ref{alg:subgradient}, we have
\begin{eqnarray}\label{klks}
\nonumber && (1-\beta_{1,t^*})\left(f(x_{t^*}) - f(\bar x_{t^*}) - \frac{\rho}{2}\|\widehat{V}_{t^*}^{1/4}(x_{t^*} - \bar x_{t^*})\|^2\right)  \\
\nonumber &= &
(1-\beta_{1,t^*})\Big(\big(f(x_{t^*})+\frac{\bar \rho}{2}\|\widehat{V}_{t^*}^{1/4}(x_{t^*} - x_{t^*})\|^2\big) \\
\nonumber
&-& \big( f(\bar x_{t^*}) + \frac{\bar \rho}{2}\|\widehat{V}_{t^*}^{1/4}(x_{t^*} - \bar x_{t^*})\|^2\big)
 + \frac{\bar \rho - \rho}{2}\|\widehat{V}_{t^*}^{1/4}(x_{t^*} - \bar x_{t^*})\|^2\bigg)  \\
\nonumber
&\geq & (1-\beta_{1,t^*})(\bar \rho-\rho)\|\widehat{V}_{t^*}^{1/4}(x_{t^*} - \bar x_{t^*})\|^2 \\
&= & \frac{ (1-\beta_{1,t^*}) (\bar \rho-\rho)}{\bar \rho^2}\|\nabla \psi_{1/\bar \rho,\widehat{V}_{t^*}^{1/2}}(x_{t^*})\|^2,
\end{eqnarray}
where the last equality follows from \eqref{eqn:grad_form}.

Finally, we substitute the lower bound \eqref{klks} into the L.H.S. of \eqref{myeqn:main_ineqs} to obtain \eqref{Projected}.
\end{proof}

Theorem ~\ref{Sthm:stochastic_sub} implies that the convergence rate of \textsc{Fema} has a dependency on $\alpha_t$ and $\bar \rho$ for weakly convex problems. The following corollary shows how to choose $\alpha_t$  and  $\bar \rho$ appropriately in the projected setting.

\begin{cor}\label{Scor:complexity}
Under the same conditions as in Theorem~\ref{Sthm:stochastic_sub}, for all $t \in (T]$, let the parameters be set to
\begin{align*}
\beta_{1,t}=\beta_1\pi^{t-1},~~~\pi \in (0,1), ‍~~~\bar \rho=2\rho, \quad \textnormal{and} \quad \alpha_t =\frac{\alpha}{\sqrt{T+1}},
\end{align*}
for some $\alpha>0$. Then, for $x_{t^*}$ generated by Algorithm~\ref{alg:subgradient}, we have
\begin{align*}
\EE \left[\|\nabla \psi_{1/(2\rho),\widehat{V}_{t^*}^{1/2}}(x_{t^*})\|^2\right]
\leq O(\frac{d}{\sqrt{T}}).
\end{align*}	
\end{cor}

\begin{proof}
Substituting $\beta_{1,t}=\beta_1\pi^{t-1}$ in  $C_{2,T}$ defined in \eqref{pj} and using Lemma~\ref{lm:infty}, we obtain
\begin{eqnarray} \label{eqn:bc23}
\nonumber
C_{2,T} &=&\frac{D_{\infty}^2 }{2}  \big(\sum_{t=0}^T  \beta_1^2 \pi^{2(t-1)}   \EE[\|\widehat{\upsilon}_{t}^{1/2}\|_1]  +  \EE[\|\widehat{\upsilon}_{T}^{1/2}\|_1]\big)\\
&\leq& \frac{d D_{\infty}^2 G_{\infty}}{2} \big(\frac{\beta_1^2}{1-2\pi} + 1\big)=: C_2.
\end{eqnarray}	
Further, for $C_{1,t}$ defined in \eqref{pj}, we have
\begin{eqnarray}\label{eqn:bc12}
\nonumber
 C_{1,t} &=&  \frac{ \sum_{k= 0}^t \tau^{t-k} \EE [\|{g_{k}}\|_1]}{ (1-\beta_1)\sqrt{(1-\beta_2)(1-\beta_3)} } \\
\nonumber
&\leq&  \frac{  \sum_{k= 0}^t \tau^{t-k} d G_{\infty}}{ (1-\beta_1)\sqrt{(1-\beta_2)(1-\beta_3)} }\\
 &\leq &  \frac{   d G_{\infty}}{(1-\tau) (1-\beta_1)\sqrt{(1-\beta_2)(1-\beta_3)} } =:C_1,
\end{eqnarray}
where the first inequality is obtained from Lemma~\ref{lm:infty} and the last inequality uses  $$\sum_{t=0}^{T} \tau^t \leq \frac{1}{1-\tau}, \qquad \text{where} \quad  \tau = \beta_1/\sqrt{\beta_2}<1. $$
Now, substituting \eqref{eqn:bc23} and \eqref{eqn:bc12} into \eqref{Projected} and using $\alpha_t=\frac{\alpha}{\sqrt{T+1}}$ and $\bar \rho=2\rho$, we get
\begin{eqnarray*}
&&\EE \left[\|\nabla \psi_{1/(2 \rho),\widehat{V}_{t^*}^{1/2}}(x_{t^*})\|^2\right] \leq  \frac{ 2\rho \Delta_\psi+4\rho^2 (\alpha^2 C_1+C_2) }{\rho (1-\beta_1)\alpha} \cdot \frac{1}{\sqrt{T+1}}.
\end{eqnarray*}
Minimizing the right-hand side of the above inequality
w.r.t. $\alpha$ yields the choice $\alpha=\sqrt{\frac{ \Delta_\psi+2\rho C_2}{2\rho C_1}}$. Thus,
\begin{align*}
\nonumber
\EE \left[\|\nabla \psi_{1/(2 \rho),\widehat{V}_{t^*}^{1/2}}(x_{t^*})\|^2\right] &\leq   \frac{ 4\sqrt{2\rho C_1}\sqrt{ \Delta_\psi+2 \rho C_2}}{(1-\beta_1)} \cdot \frac{1}{\sqrt{T+1}} \\
 &= O(\frac{d}{\sqrt{T}}).
\end{align*}
\end{proof}

\begin{rem}
The above corollary shows that \textsc{Fema} attains an $ O(d/\sqrt{T})$ rate of complexity which matches the ones obtained in \cite{davis2019stochastic,zaheer2018adaptive} for nonsmooth and smooth problems, respectively. However, in contrast to \cite{davis2019stochastic}, \textsc{Fema} uses adaptive learning rates which can significantly accelerate the convergence and improves the performance for sparse problems. Further, in comparison to results obtained in \cite{zaheer2018adaptive}, our analysis does not require batching while expand the class of functions for which the convergence rate of the stochastic gradient methods is known.
\end{rem}

\begin{rem}
It is worth mentioning that our assumption $\|G(x ; \xi)\|_{\infty} \leq G_{\infty}$ is slightly stronger than the assumption $\|G(x ; \xi)\| \leq G_{2}$ used in \cite{davis2019stochastic,duchi2018stochastic}. Indeed, since $\|\cdot\|_{\infty} \leq \|\cdot\| \leq \sqrt{d}\|\cdot\|_{\infty}$, the latter assumption implies the former with $G_\infty = G_2$. However, $G_\infty$ will be tighter than $G_2$ by a factor of $\sqrt{d}$ when each coordinate of $G(x; \xi)$ almost equals to each other. Furthermore, the assumption $\|G(x ; \xi)\|_{\infty} \leq G_{\infty}$  is crucial for the convergence analysis of adaptive subgradient
methods in the nonconvex setting and it has been widely considered in the literature; see, for instance \cite{zaheer2018adaptive,ward2018adagrad,nazari2019dadam,chen2018convergence,zhou2018convergence}.
\end{rem}

Next, we consider a more general composite problem \eqref{eqn:problem_class} and establish the convergence of Algorithm~\ref{alg:subgradient} for solving such stochastic problems. To do so, we first provide the following lemma.

\begin{lem}\label{slem:prox_ident}
For  $\bar{\rho} > \rho$, let
\begin{eqnarray*}
\bar x_t =\textnormal{prox}_{\psi/\bar \rho,Q}(x_t) \quad \textnormal{and} \quad
\bar{\gamma}_t =\EE_t[G(\bar{x}_t,\xi_t)]\in\partial f(\bar{x}_t).
\end{eqnarray*}
Then, under Assumption~\ref{it1}, we have
\begin{equation*}
\bar x_t = \textnormal{prox}_{\alpha_t h,Q}\big(\alpha_t \bar \rho x_t - \alpha_t Q^{-1}  \bar{\gamma}_t + (1-\alpha_t \bar \rho)\bar x_t\big).
\end{equation*}
\end{lem}
\begin{proof}
The proof is similar to \cite[Lemma~3.2]{davis2019stochastic}.
By the definition of the scaled proximal map in \eqref{prox}, we have
\begin{equation*}\label{zaz}
  \bar{x}_t=\arg \min_{y}\{f(y)+h(y)+ \frac{\bar \rho\|{Q^{1/2}}(y-x_t) \|^2} {2} \}.
\end{equation*}
Now, by the optimality condition of the above minimization problem, we have
$$\bar \rho Q(x_t - \bar x_t)  \in \partial h(\bar x_t) + \bar{\gamma}_t,$$
which implies that
\begin{eqnarray*}
&& \alpha_t \bar \rho Q(x_t - \bar x_t)  \in \alpha_t\partial h(\bar x_t) + \alpha_t \bar{\gamma}_t
\\
&\iff& \alpha_t \bar \rho x_t - \alpha_t Q^{-1} \bar{\gamma}_t+ (1-\alpha_t \bar \rho)\bar x_t  \in \bar x_t+\alpha_t Q^{-1} \partial h(\bar x_t) \\
	&\iff& \bar x_t = \textnormal{prox}_{\alpha_t h,Q}\big(\alpha_t \bar \rho x_t - \alpha_t Q^{-1}  \bar{\gamma}_t + (1-\alpha_t \bar \rho)\bar x_t\big).
	\end{eqnarray*}
\end{proof}

\begin{thm}[\textbf{proximal \textsc{Fema}}]\label{Sthm:stochastic_sub2}
Suppose Assumptions~\ref{it1}--\ref{it2} hold, $\|G(x,\cdot)\|_{\infty} \leq G_{\infty}$ and $\|x-y\|_{\infty} \leq D_{\infty}$ for all $x, y\in \dom h$. Further, for all $ t \in (T]$, let
\begin{align*}
&& 0 \leq \beta_{1,t}\leq\beta_1, ~~~\{\beta_i\}_{i=1}^3  \in [0,1),   ~~~\tau = \frac{\beta_1}{\sqrt{\beta_2}}<1, ~~~ \bar \rho \in \big(\rho,2\rho\big],  \quad \textnormal{and} \quad 0<\alpha_t \leq \frac{1}{\bar \rho}.
\end{align*}
Then, for $x_{t^*}$ returned by Algorithm~\ref{alg:subgradient}, we have
\begin{eqnarray}\label{fd}
\nonumber \EE \left[\|\nabla \psi_{1/\bar \rho,\widehat{V}_{t^*}^{1/2}}(x_{t^*})\|^2\right]
\leq \frac{ \bar \rho \Delta_{\psi}+ \bar \rho^2 ( \sum_{t=0}^{T} \alpha_t^2 C_{1,t}+  C_{2,T} )}{(\bar \rho-\rho)\sum_{t=0}^T \alpha_t },
\end{eqnarray}
where $\Delta_{\psi}=\psi_{1/\bar \rho,\widehat{V}_{-1}^{1/2}}(x_{0}) - \psi^*$,
\begin{align}\label{prox:pj}
\nonumber
C_{1,t} &= \frac{2 \sum_{k=0}^t\tau^{t-k} \EE [\|g_{k}\|_1]}{(1-\beta_1)\sqrt{(1-\beta_2)(1-\beta_3)}} + \frac{dG_{\infty}}{\sqrt{(1-\beta_2)(1-\beta_3)}}+\frac{dG_\infty^2}{\lambda_{\min}(\widehat{V}_{t}^{1/2})} , \quad \textnormal{and}\\
C_{2,T} &= \frac{D_{\infty}^2}{2} (\EE[ \|\widehat{\upsilon}_{T}^{1/2} \|_1]+\sum_{t=0}^T\beta_{1,t}^2\EE[\|\widehat{\upsilon}_{t}^{1/2}\|_1 ]).
\end{align}
\end{thm}
\begin{proof}
 Let $\delta_t=1-\alpha_t\bar  \rho$ and  $\bar{\gamma}_t =\EE_t[G(\bar{x}_t,\xi_t)]$. It follows from Lemma~\ref{slem:prox_ident} that
\begin{eqnarray}\label{jk}
\nonumber
 && \EE_t[  \|\widehat{V}_{t}^{1/4}(x_{t+1} - \bar x_t)\|^2]\\
\nonumber
&=& \EE_t[ \|\widehat{V}_t^{1/4}\big(x_{t+1}-\textnormal{prox}_{\alpha_t h,\widehat{V}_t^{1/2}}(\alpha_t \bar \rho x_t - \alpha_t\widehat{V}_t^{-1/2}\bar \gamma_t + (1-\alpha_t \bar \rho)\bar x_t)\big)\|^2]\\ \nonumber
	&\leq& \EE_t[\|\widehat{V}_{t}^{1/4}\big(x_t - \alpha_t \widehat{V}_t^{-1/2}m_t - (\alpha_t \bar \rho x_t - \alpha_t\widehat{V}_t^{-1/2}\bar \gamma_t + (1-\alpha_t \bar \rho)\bar x_t)\big)\|^2] \\\nonumber
	&=& \EE_t[\|\widehat{V}_{t}^{1/4}\delta_t(x_t - \bar x_t) - \alpha_t( \widehat{V}_{t}^{-1/4}m_t -\widehat{V}_{t}^{-1/4} \bar \gamma_t)\|^2]\\\nonumber
	&\leq& 2 \alpha_t^2\big( \EE_t[\| \widehat{V}_{t}^{-1/4} \bar \gamma_t\|^2] +   \EE_t[\| \widehat{V}_{t}^{-1/4}m_t\|^2]\big) \\
&+&  2 \delta_t\alpha_t \EE_t[\dotp{(\bar x_t-x_t ), m_t - \bar \gamma_t}] + \delta_t^2\EE_t[\|\widehat{V}_{t}^{1/4}(x_t - \bar x_t )\|^2] ,
	\end{eqnarray}
where the first inequality follows from Lemma~\ref{mc}(ii).

Now, by the definition of $m_t$ in Algorithm~\ref{alg:subgradient}, we have
\begin{align} \label{eq:thm0122}
\EE_t[ \dotp{\bar x_t-x_t , m_t- \bar \gamma_t}]
= \EE_t[\dotp{\bar x_t-x_t ,\beta_{1,t}(m_{t-1}-g_t)} ] + \EE_t[ \dotp{\bar x_t-x_t , g_t- \bar \gamma_t}].
\end{align}
Next, we multiple and divide the first term by $ \alpha_t^{1/2} \widehat{V}_{t}^{-1/4}$ and use the Cauchy-Schwarz inequality to obtain
\begin{subequations}
\begin{eqnarray}\label{eq:thm4.2.2}
\nonumber && \EE_t[\dotp{\bar x_t-x_t ,\beta_{1,t} (m_{t-1}-g_t)}]\\
\nonumber
&\leq& \frac{\alpha_t}{2}\EE_t[\|\widehat{V}_{t}^{-1/4}  (m_{t-1}-g_t)\|^2]+ \frac{\beta_{1,t}^2}{2\alpha_t} \EE_t[\|\widehat{V}_{t}^{1/4}(\bar{x}_t-x_t)\|^2]\\
&\leq & \alpha_t\EE_t[\|\widehat{V}_{t-1}^{-1/4}  m_{t-1}\|^2] +\alpha_t\EE_t[\|\widehat{V}_{t}^{-1/4} g_t\|^2]+ \frac{\beta_{1,t}^2}{2\alpha_t} \EE_t[\|\widehat{V}_{t}^{1/4}(\bar{x}_t-x_t)\|^2],
\end{eqnarray}
where the second inequality follows since $(\widehat{V}_{t})_{ii}\geq (\widehat{V}_{t-1})_{ii}$ for all $i\in [d]$.

Let $\gamma_t =\EE_t[G(x_t,\xi_t)]$. Assumption~\ref{it1} in conjunction with Lemma~\ref{lem:weak:hyp}(iii) yields
\begin{align}\label{eq:thm4.2.3}
-\EE_t[ \dotp{x_t-\bar x_t , g_t- \bar \gamma_t}]=-\dotp{x_t - \bar x_t,\gamma_t- \bar{\gamma}_t} \leq \rho\|\widehat{V}_{t}^{1/4}(x_t - \bar x_t)\|^2.
\end{align}
\end{subequations}
Substituting \eqref{eq:thm4.2.2} and \eqref{eq:thm4.2.3} into \eqref{jk} and using our assumption $ \delta_t \leq 1$, we obtain
\begin{eqnarray}\label{befzz}
\nonumber
 && \EE_t[  \|\widehat{V}_{t}^{1/4}(x_{t+1} - \bar x_t)\|^2]\\
\nonumber &\leq&  \|\widehat{V}_{t-1}^{1/4}(x_t-\bar{x}_t)\|^2 -\big(2\alpha_t(\bar \rho -\rho)+\alpha_t^2 \bar \rho (2\rho-\bar\rho)\big) \EE_t[\|\widehat{V}_{t}^{1/4}(x_t - \bar x_t )\|^2]
\\& +& \EE_t[\underbrace{\|(\widehat{V}_{t}^{1/4}-\widehat{V}_{t-1}^{1/4})( x_{t}-\bar x_t)\|^2+\beta_{1,t}^2 \|\widehat{V}_{t}^{1/4}(\bar{x}_t-x_t)\|^2}_{C_{1,2,t}}]  \nonumber
 \\\nonumber&+&\alpha_t^2 \big( \EE_t[ \underbrace{2\| \widehat{V}_{t}^{-1/4}m_t\|^2+2 \|\widehat{V}_{t-1}^{-1/4}  m_{t-1}\|^2+2 \| \widehat{V}_{t}^{-1/4} \bar \gamma_t\|^2 +2 \|\widehat{V}_{t}^{-1/4}  g_t\|^2}_{C_{1,1,t}}]
\big)
\\ \nonumber &\leq & \|\widehat{V}_{t-1}^{1/4}(\bar{x}_t-x_t)\|^2 - 2\alpha_t\big(\bar \rho-\rho\big)\EE_t[ \|\widehat{V}_{t}^{1/4}(x_t - \bar x_t) \|^2 ]
\\&+&\alpha_t^2 \EE_t[C_{1,1,t}]+ \EE_t [C_{1,2,t}],
\end{eqnarray}
where the first inequality uses \eqref{eq:thm0012} and the second inequality follows since  $\rho<\bar \rho \leq 2\rho$.
We combine \eqref{befzz} with \eqref{prox} to get
\begin{eqnarray*}\label{mn}
\nonumber
&&\EE_t[\psi_{1/\bar \rho,\widehat{V}_t^{1/2}}(x_{t+1})]
\\
&\leq & \EE_t[ \psi(\bar x_t) + \frac{\bar \rho}{2} \|\widehat{V}_{t}^{1/4}(x_{t+1} - \bar x_t)\|^2]  \\ \nonumber
&\leq&
\psi_{1/\bar \rho,\widehat{V}_{t-1}^{1/2}}(x_{t})
 + \frac{\bar \rho}{2}\Big(- 2\alpha_t\big(\bar \rho-\rho\big)\EE_t[ \|\widehat{V}_{t}^{1/4}(x_t - \bar x_t) \|^2 ] +\alpha_t^2 \EE_t[C_{1,1,t}]+ \EE_t [C_{1,2,t}]\Big).
\end{eqnarray*}
Taking expectations on both sides of the above inequality w.r.t. $\xi_0,\xi_1,\ldots,\xi_{t-1}$ and using the tower rule, we get
\begin{eqnarray}\label{pkpp1}
\nonumber
&& \bar \rho\alpha_t\big(\bar \rho-\rho\big) \sum_{t=0}^{T}\EE [\|\widehat{V}_{t}^{1/4}(x_t - \bar x_t) \|^2
] \\
&\leq &  \psi_{1/\bar \rho,\widehat{V}_{-1}^{1/2}}(x_{0}) - \EE [\psi_{1/\bar \rho,\widehat{V}_{T}^{1/2}}(x_{T+1})] +  \frac{\bar \rho }{2}\sum_{t=0}^{T} \Big(\alpha_t^2 \EE[C_{1,1,t}]+\EE[C_{1,2,t}]\Big).
\end{eqnarray}
Next, we upper bound each term on the R.H.S. of the above inequality.  To bound the third term, it follows from Lemma~\ref{lem:spar} that
\begin{eqnarray*}
\nonumber
&&\frac{\bar \rho }{2}\sum_{t=0}^{T}\alpha_t^2 \EE[C_{1,1,t}] \leq    \frac{\bar{\rho}}{2}  \sum_{t=0}^{T}  \frac{4   \alpha_t^2  \sum_{i=1}^d \sum_{k=0}^t{\tau}^{t-k} \EE[ |{g_{k,i}}|] }{{(1-\beta_1)\sqrt{(1-\beta_2)(1-\beta_3)}}} \\\nonumber&+&
\frac{\bar{\rho}}{2}  \sum_{t=0}^{T}  \frac{2   \alpha_t^2  \sum_{i=1}^d  \EE[ |{g_{t,i}}|] }{{\sqrt{(1-\beta_2)(1-\beta_3)}}} +
  \frac{\bar{\rho}}{2}  \sum_{t=0}^{T} 2\alpha_t^2 \EE[\| \widehat{V}_{t}^{-1/4} \bar \gamma_t\|^2] \\
\nonumber
 & \leq &  \bar{\rho}\sum_{t=0}^{T} \frac{ 2 \alpha_t^2  \sum_{k=0}^t\tau^{t-k} \EE [\|g_{k}\|_1]}{(1-\beta_1)\sqrt{(1-\beta_2)(1-\beta_3)}}
 +   \frac{ \bar{\rho}  \sum_{t=0}^{T} \alpha_t^2  d G_{\infty} }{{\sqrt{(1-\beta_2)(1-\beta_3)}}}+   \frac{\bar{\rho}  \sum_{t=0}^{T} \alpha_t^2  dG_\infty^2}{\lambda_{\min}(\widehat{V}_{t}^{1/2})} \\
 &= &  \bar{\rho}\sum_{t=0}^{T} \alpha_t^2  C_{1,t},
\end{eqnarray*}
where $C_{1,t}$ is defined \eqref{prox:pj}.

Further, by the same argument as in \eqref{pp}, we have
\begin{eqnarray}\label{fppz}
\nonumber
\frac{\bar \rho }{2}\sum_{t=0}^{T} \EE[C_{1,2,t}]
&\leq& \frac{\bar \rho D_{\infty}^2}{2} (\EE[\| \widehat{\upsilon}_{T}^{1/2}\|_1]+\sum_{t=0}^T\beta_{1,t}^2\EE[\|\widehat{\upsilon}_{t}^{1/2}\|_1 ]) = \bar{\rho} C_{2,T},
\end{eqnarray}
where $C_{2,T}$ is defined \eqref{prox:pj}.

We combine \eqref{grad} with \eqref{pkpp1} to obtain
\begin{align*}
&\frac{1}{\bar \rho^2}\nonumber \|\nabla \psi_{1/\bar \rho,\widehat{V}_{t^*}^{1/2}}(x_{t^*})\|^2
 =\EE [\|\widehat V_{t^*}^{1/4}(x_{t^*} - \bar x_{t^*})\|^2]\\
&=\frac{1}{\sum_{t=0}^T \alpha_t}\sum_{t=0}^T \alpha_t \EE [\|\widehat V_{t}^{1/4}(x_{t} - \bar x_{t})\|^2]\\
\nonumber
&\leq \frac{1}{\bar \rho\big(\bar \rho-\rho\big)\sum_{t=0}^{T}\alpha_t}\Big(\psi_{1/\bar \rho,\widehat{V}_{-1}^{1/2}}(x_{0}) - \psi^*  + \bar \rho  \sum_{t=0}^{T} \alpha_t^2 C_{1,t} + \bar{\rho} C_{2,T}\Big),
\end{align*}
where the last inequality follows since $ \psi_{1/\bar \rho,\widehat{V}_{T}^{1/2}}(x_{T+1})\geq \psi^*$ and  $\psi^*$ is finite.

\end{proof}

The following corollary, analogous to the Corollary \ref{Scor:complexity}, shows how to choose $\alpha_t$  and  $\bar \rho$  appropriately in the proximal setting.

\begin{cor}
Under the same conditions as in Theorem~\ref{Sthm:stochastic_sub2}, for all $t \in (T]$, let the parameters be set to
\begin{align*}
\beta_{1,t}&=\beta_1\pi^{t-1},~~~ \pi\in (0,1), ~~~\bar \rho=2\rho, ~~~ \alpha_t =\frac{\alpha}{\sqrt{T+1}},  \quad \textnormal{and} \quad \alpha \in (0,\frac{1}{2\rho}].
\end{align*}
Then, for $x_{t^*}$ returned by Algorithm~\ref{alg:subgradient}, we have
\begin{align*}
\EE \left[\|\nabla \psi_{1/(2\rho),\widehat{V}_{t^*}^{1/2}}(x_{t^*})\|^2\right]
\leq O(\frac{d}{\sqrt{T}}).
\end{align*}	
\end{cor}
\begin{proof}
Substituting $\beta_{1,t}=\beta_1\pi^{t-1}$ in  $C_{2,T}$ defined in \eqref{prox:pj} and using Lemma~\ref{lm:infty}, we obtain
\begin{eqnarray*}
\nonumber
C_{2,T} &=&\frac{D_{\infty}^2 }{2} (\EE[\|\widehat{\upsilon}_{T}^{1/2}\|_1]+\sum_{t=0}^T\beta_{1,t}^2\EE[\|\widehat{\upsilon}_{t}^{1/2}\|_1])
\\\nonumber&\leq&  \frac{d D_{\infty}^2 G_{\infty}}{2}(\sum_{t=0}^T\beta_1^2\pi^{2(t-1)}+1)\\
&\leq&\frac{d D_{\infty}^2 G_{\infty}}{2}(\frac{\beta_1^2}{1-2\pi}+1)=: C_2.
\end{eqnarray*}	
Also, for the term $C_{1,t}$ defined in \eqref{prox:pj}, we have
\begin{eqnarray*}
\nonumber
 C_{1,t} &=&  (\frac{ 2\sum_{k= 0}^t \tau^{t-k} \EE [\|{g_{k}}\|_1]}{(1-\beta_1)  }+dG_{\infty}) \frac{1}{\sqrt{(1-\beta_2)(1-\beta_3)}} + \frac{ dG_\infty^2}{\lambda_{\min}(\widehat{V}_t^{1/2})} \\
\nonumber
&\leq&  (\frac{ 2 \sum_{k= 0}^t \tau^{t-k} }{(1-\beta_1) }+1) \frac{dG_{\infty}}{\sqrt{(1-\beta_2)(1-\beta_3)}} +\frac{ dG_\infty^2}{\lambda_{\min}(\widehat{V}_t^{1/2})}\\
 &\leq &  (\frac{ 2 }{(1-\tau) (1-\beta_1) }+1)\frac{dG_{\infty}}{\sqrt{(1-\beta_2)(1-\beta_3)}} +\frac{   dG_\infty^2}{\lambda_{\min}(Q)} =:C_1,
\end{eqnarray*}
where the first inequality is obtained from Lemma~\ref{lm:infty} and the last inequality uses  $$\sum_{t=0}^{T} \tau^t \leq \frac{1}{1-\tau}, \qquad \text{where} \quad  \tau = \frac{\beta_1}{\sqrt{\beta_2}}<1. $$
The reminder of proof is similar to that of Corollary~\ref{Scor:complexity}.
\end{proof}
\subsection{ Zeroth-order EMA-type method (\textsc{Zema})}

Our goal in this section is to specialize the EMA-type method to deal with the situation when only stochastic zeroth-order information of the  function $f(x)$ is available. Henceforth, we consider the following stochastic optimization problem
\begin{equation} \label{eqn:zproblem_class}
\min_{x \in \R^d} \psi(x)= f(x) + h(x), \quad \textnormal{where} \quad  \EE_{\xi}[F(x,\xi)] = f(x).
\end{equation}
Here,  $f$, $h$, and $\psi$ satisfy Assumption~\ref{it1}.
%
%

In addition, similar to the first order setting, we let $G(x,\xi) \in \partial_x F(x,\xi)$ satisfy Assumption \ref{it2}.
Under these assumptions, Algorithm~\ref{alg:subgradientzero} depicts a variant of zeroth-order algorithm where at each iteration $t$, the subgradient estimator $g_{\mu,t}$ is constructed as follows
$$
g_{\mu,t}:=G_\mu(x_t,\xi_{t},u_{t}),
$$
where
\begin{equation}\label{ppl}
  G_\mu(x_t,\xi_{t},u_{t})=\frac{d}{\mu}\big[F(x_t+\mu u_{t},\xi_{t})-F(x_t,\xi_{t})\big]u_{t}.
\end{equation}
Here, $\mu > 0$ is a smoothing parameter and $\{u_{t}\}_{t=0}^T$ are  $d$-dimensional random vectors uniformly distributed over the unit Euclidean ball \cite{shalev2012online}.

\begin{algorithm}[t]
\caption{Zeroth-order EMA-type method (\textsc{Zema})}\label{alg:subgradientzero}
\SetKwInput{Input}{input~}
\SetKwInput{Output}{output~}
\SetKwInput{optionone}{option~I~}
\SetKwInput{optiontwo}{option~II~}
\Input{Initial point $x_0 \in \dom h$, number of iterations $T$,  stepsizes $\{\alpha_t\}_{t=0}^T\subset\R_+$, decay parameters $\{\beta_{1,t}\}_{t=0}^T, \beta_2, \beta_3 \in [0,1)$, smoothing parameter $\mu >0$ and a vector $q \in \mathbb{R}_+^d$ satisfying Assumption~\ref{it5}\;}
\BlankLine
Initialize $m_{\mu,-1}= \upsilon_{\mu,-1}=0$ and $\widehat{\upsilon}_{\mu,-1}= q $\;
\For{ $0\leftarrow t : T$}{ \label{forins}
Compute $G_\mu(x_t,\xi_{t},u_{t})$ by \eqref{ppl}\;
$g_{\mu,t}=G_\mu(x_t,\xi_{t},u_{t})$\;
$m_{\mu,t}= \beta_{1,t} m_{\mu,t-1}+(1-\beta_{1,t})g_{\mu,t}$\; 
$\upsilon_{\mu,t}= \beta_2\upsilon_{\mu,t-1}+(1-\beta_2) g_{\mu,t}^2$\;
 $\widehat{\upsilon}_{\mu,t}= \beta_{3}\widehat{\upsilon}_{\mu,t-1}+(1-\beta_{3}) \max (\widehat{\upsilon}_{\mu,t-1},\upsilon_{\mu,t})$ and $\widehat{V}_{\mu,t}= \text{diag}(\widehat{\upsilon}_{\mu,t})$\;
 $ x_{t+1}=\textnormal{prox}_{\alpha_th,\widehat{V}_{\mu,t}^{1/2}}\big(x_{t} - \alpha_t\widehat{V}_{\mu,t}^{-1/2} m_{\mu,t}\big)$\label{m11} \;
}
\BlankLine
\Output{Choose $x_{t^*}$ from $\{x_t\}_{t=0}^{T}$ with probability $\mathbb{P}(t^*=t)=\alpha_t/\sum_{t=0}^T \alpha_t$.}
\end{algorithm}

Compared with \textsc{Fema} algorithm, we can see at the $t^{th}$ iteration, the $\textsc{Zema}$ algorithm simply replaces the stochastic subgradient $ G(x_t,\xi_{t})$ by the approximated stochastic gradient $G_\mu(x_t,\xi_{t},u_{t})$. We note that unlike the first-order stochastic gradient estimate, the zeroth order gradient estimate \eqref{ppl} is a biased approximation to the true subgradient of $f$. Instead,  it becomes unbiased to the gradient of the randomized smoothing function $f_\mu$ \cite{gao2018information}:
\begin{equation} \label{rand_smooth_func}
f_{\mu}(x) = \frac{1}{\mathcal{V}(d)} \int_{u \in B} f(x+\mu u)du =\bbe_{\{u\sim U_b\}}\left[f(x+\mu u)\right],
\end{equation}
where $U_b$ is the uniform distribution over the unit Euclidean ball; $B$ is the unit ball; and $\mathcal{V}(d)$ is the volume of the unit ball in $\Bbb{R}^d$.

This relation implies that we can estimate gradient of $f_\mu$ by only using evaluations of $f$ as formulated in \eqref{ppl} since
\begin{equation}\label{eqn:ex_G_mu}
 \gamma_{\mu}= \EE_{\xi, u}[G_{\mu}(x, \xi, u)] = \EE_u \left[ \EE_{\xi}[G_{\mu}(x, \xi, u)|u] \right] = \nabla f_{\mu}(x).
\end{equation}

Next, we discuss some important properties of this subgradient estimator for $(\rho,Q)$-weakly convex functions.

\begin{lem}\label{M2}
If $f$ is $(\rho,Q)$-weakly convex, then $f_{\mu}$ is  also $(\rho,Q)$-weakly convex.
\end{lem}
\begin{proof}
It follows from the definition of $f_{\mu}(x)$ in \eqref{rand_smooth_func} that
\begin{align}\label{fff0}
  \nonumber f_{\mu}(x)&+\frac{\rho \mu^2}{2}\EE_{u}\Big[\|Q^{1/2}u\|^2\Big]+\frac{\rho}{2}\|Q^{1/2}x\|^2 \\
  \nonumber &= \frac{1}{\mathcal{V}(d)}\int_{B} f(x+\mu u) \,du\\\nonumber
  &+\frac{\rho \mu^2}{2{\mathcal{V}(d)}} \int_{B}\|Q^{1/2}u\|^2du+ \frac{\rho}{2}\|Q^{1/2}x\|^2 \\ \nonumber&= \frac{1}{\mathcal{V}(d)}\int_{B} f(x+\mu u) \,du\\
 \nonumber
&+\frac{\rho}{2{\mathcal{V}(d)}}\Big( \int_{B} \big(\|Q^{1/2}x\|^2+2\mu Q \langle x,u \rangle+\mu^2\|Q^{1/2}u\|^2\big) du\Big)  \\
   &= \frac{1}{\mathcal{V}(d)}\int_{B} \Big(f(x+\mu u)+\frac{\rho}{2}\|Q^{1/2}(x+\mu u)\|^2\Big) du,
\end{align}
where the first equality follows from \eqref{rand_smooth_func}.

We note that the function on the R.H.S. of \eqref{fff0} is convex because $f$ is $(\rho,Q)$-weakly convex. Thus,
$$x\mapsto f_{\mu}(x)+\frac{\rho \mu^2}{2}\EE_{u}\Big[\|Q^{1/2}u\|^2\Big]+\frac{\rho}{2}\|Q^{1/2}x\|^2$$ is
convex. Now, by Definition~\ref{def:weak}, the functions $f_{\mu}+\frac{\rho \mu^2}{2}\EE_{u}\Big[\|Q^{1/2}u\|^2\Big]$ and $f_{\mu}$  are $(\rho,Q)$-weakly convex.
\end{proof}

%

We next look at the consequences of our results in the zeroth-order setting where we only have access to noisy function values. Similar to Theorem~\ref{Sthm:stochastic_sub}, we first present a simplified argument for $h$ being the indicator function of a closed convex set $\cX$.  The following theorem establishes the convergence of stochastic zeroth order methods including \textsc{Zema} for solving constrained problem \eqref{eqn:zproblem_class}.

\begin{thm}[\textbf{projected \textsc{Zema}}]\label{thm:stochastic_subzero}
Suppose Assumptions~\ref{it1}--\ref{it2} hold and $\|x-y\|_{\infty} \leq D_{\infty}$ for all $x, y\in \cX$. Further, for all $t \in (T]$, let
\begin{align*}
 0 \leq \beta_{1,t}\leq\beta_1, ~~~\{\beta_i\}_{i=1}^3  \in [0,1), ~~~\tau = \frac{\beta_1}{\sqrt{\beta_2}}<1,~~~\bar{\rho} >\rho, \quad \textnormal{and} \quad  \mu>0.
\end{align*}
Then, for $x_{t^*}$ generated by Algorithm~\ref{alg:subgradientzero}, we have
\begin{eqnarray}\label{thm}
\nonumber
\EE \left[\|\nabla \psi_{1/\bar \rho,\widehat{V}_{\mu,t^*}^{1/2}}(x_{t^*})\|^2\right] &\leq&
 \frac{ \bar \rho \Delta_{\psi}+ \bar \rho^2 ( \sum_{t=0}^{T} \alpha_t^2 C_{1,\mu,t}+  C_{2,\mu,T}) }{(\bar \rho-\rho)(1-\beta_1)\sum_{t=0}^T \alpha_t } \\
 &+& \frac{\bar{\rho}^2 C_{3,\mu,T}}{\bar{\rho}-\rho}.
\end{eqnarray}
Here, $\Delta_{\psi}=\psi_{1/\bar \rho,\widehat{V}_{\mu,-1}^{1/2}}(x_{0}) - \psi^*$,
\begin{eqnarray}\label{eqn:c2}
\nonumber
C_{1,\mu,t} &=&\frac{  \sum_{k=0}^t\tau^{t-k} \EE [\|g_{\mu,k}\|_1]}{(1-\beta_1)\sqrt{(1-\beta_2)(1-\beta_3)}},\\
\nonumber
C_{2,\mu,T} &=& \frac{ D_{\infty}^2}{2} \big(\sum_{t=0}^T \beta_{1,t}^2 \EE[ \|\widehat{\upsilon}_{\mu,t}^{1/2} \|_1] +\EE[ \|\widehat{\upsilon}_{\mu,T}^{1/2} \|_1]\big),\\
C_{3,\mu,T} &=&  \EE \big[  \max_{t \in (T]}|f(\bar{x}_t)-f_{\mu}(\bar{x}_t)| \big]+\EE \big[  \max_{t \in (T]}|f(x_t)-f_{\mu}(x_t)| \big].
\end{eqnarray}
\end{thm}
\begin{proof}
Let
\begin{eqnarray*}
 \S_t  \equiv  (\xi_t,u_t), \quad \text{and} \quad \bar x_t  = \textnormal{prox}_{\psi/\bar \rho,\widehat{V}_{\mu,t}^{1/2}}(x_t),\quad \textnormal{for\,\,all} \quad  t \in (T].
\end{eqnarray*}
Also, let $\EE_{\S_t }[\cdot]$ to denote the expectation conditioned on all the realizations $\S_0, \S_1, \ldots , \S_{t-1}$. For the sequence $\{x_t\}_{t=0}^T$ generated by Algorithm~\ref{alg:subgradientzero}, we have
\begin{eqnarray*}
\nonumber
&& \EE_{\S_t}[ \psi_{1/\bar \rho,\widehat{V}_{\mu,t}^{1/2}}(x_{t+1})] \leq \EE_{\S_t}[ f(\bar x_t) + \frac{\bar \rho}{2} \|\widehat{V}_{\mu,t}^{1/4}(x_{t+1} - \bar x_t)\|^2]\notag\\
&=& f(\bar x_t) + \frac{\bar \rho}{2}\EE_{\S_t}[ \|\widehat{V}_{\mu,t}^{1/4}[\Pi_{\cX,\widehat{V}_{\mu,t}^{1/2}}(x_{t} - \alpha_t \widehat{V}_{\mu,t}^{-1/2}m_{\mu,t}) -  \Pi_{\cX,\widehat{V}_{\mu,t}^{1/2}}(\bar x_t)]\|^2] \nonumber\\
&\leq & f(\bar x_t) + \frac{\bar \rho}{2}\EE_{\S_t}[ \|\widehat{V}_{\mu,t}^{1/4}(x_{t}  - \bar x_t- \alpha_t \widehat{V}_{\mu,t}^{-1/2}m_{\mu,t})\|^2] \nonumber\\
\nonumber
&\leq & f(\bar x_t) + \frac{\bar \rho}{2}  \EE_{\S_t}[\|\widehat{V}_{\mu,t}^{1/4}(x_{t} - \bar x_t)\|^2]
\\&+& \bar \rho\alpha_t\EE_{\S_t} [\dotp{\bar x_t-x_t , m_{\mu,t}}]
+  \frac{\alpha_t^2\bar \rho}{2}\EE_{\S_t}[ \|\widehat{V}_{\mu,t}^{-1/4}m_{\mu,t}\|^2],
\end{eqnarray*}
where the first inequality follows from \eqref{prox} and the second inequality is obtained from Lemma~\ref{mc}(ii).

Similar to what is done in the proof of Theorem~\ref{Sthm:stochastic_sub} and by the same argument as in \eqref{eq:thm012}-\eqref{eq:thm0012} we have
\begin{eqnarray*}
\nonumber
&&\EE_{\S_t}[ \psi_{1/\bar \rho,\widehat{V}_{\mu,t}^{1/2}}(x_{t+1})]  \leq
\psi_{1/\bar \rho,\widehat{V}_{\mu,t-1}^{1/2}}(x_{t})+\bar \rho(1-\beta_{1,t}) \alpha_t \EE_{\S_t}[ \dotp{\bar x_t-x_t ,g_{\mu,t}}]
\\
\nonumber&+& \frac{ \bar \rho}{2}  \EE_{\S_t}  [\|(\widehat{V}_{\mu,t}^{1/4}-\widehat{V}_{\mu,t-1}^{1/4})(x_{t} - \bar x_t)\|^2]+\frac{\beta_{1,t}^2 \bar \rho}{2}  \EE_{\S_t}  [\|\widehat{V}_{\mu,t}^{1/4}(x_{t} - \bar x_t)\|^2]
\\ &+&
 \frac{ \bar \rho\alpha_t^2}{2}  \big( \EE_{\S_t}[ \|\widehat{V}_{\mu,t}^{-1/4}m_{\mu,t}\|^2] +  \|\widehat{V}_{\mu,t-1}^{-1/4} m_{\mu,t-1}\|^2 \big)  .
\end{eqnarray*}
Now, let  $ \gamma_{\mu,t} := \EE_{\S_t}[G_{\mu}(x_t, \xi_t,u_t)]$ for all $t\in (T]$.
Taking expectations on both sides of the above inequality w.r.t. $\S_0,\S_1,\ldots,\S_{t-1}$, and using the law of total expectation, we obtain
\begin{eqnarray}\label{cx}
\nonumber
&& \EE[ \psi_{1/\bar \rho,\widehat{V}_{\mu,T}^{1/2}}(x_{T+1})] -\psi_{1/\bar \rho,\widehat{V}_{\mu,-1}^{1/2}}(x_{0})
\leq \bar \rho\sum_{t = 0}^T(1-\beta_{1,t})\alpha_t\EE [  \dotp{\bar x_t - x_t , \gamma_{\mu,t} }]
\\\nonumber
&+&\frac{\bar \rho}{2} \sum_{t = 0}^T \EE[\|(\widehat{V}_{\mu,t}^{1/4}-\widehat{V}_{\mu,t-1}^{1/4})(x_{t} - \bar x_t)\|^2]   + \frac{\bar \rho}{2}  \sum_{t=0}^{T}\beta_{1,t}^2\EE[\|\widehat{V}_{\mu,t}^{1/4}(x_{t} - \bar x_t)\|^2]\\
&+&\frac{\bar \rho}{2} \sum_{t=0}^{T}\alpha_t^2 \big( \EE[ \|\widehat{V}_{\mu,t}^{-1/4}m_{\mu,t}\|^2] +  \EE[ \|\widehat{V}_{\mu,t-1}^{-1/4} m_{\mu,t-1}\|^2] \big).
\end{eqnarray}
Next, we upper bound each term on the R.H.S. of the above inequality. By the same argument as in \eqref{eq:fcccthmi1}-\eqref{pp}, we have
\begin{eqnarray}\label{ppz}
\nonumber
&& \frac{\bar{\rho}}{2} \sum_{t=0}^T \EE[ \|(\widehat{V}_{\mu,t}^{1/4}-\widehat{V}_{\mu,t-1}^{1/4})(x_{t} - \bar x_t)\|^2] +\frac{\bar{\rho}}{2} \sum_{t=0}^T  \beta_{1,t}^2 \EE[ \|\widehat{V}_{\mu,t}^{1/4}(x_{t} - \bar x_t)\|^2] \\
\nonumber
&\leq & \bar{\rho} C_{2,\mu,T},\\
\nonumber
&& \frac{\bar{\rho}}{2}  \sum_{t=0}^{T}  \alpha_t^2   \big( \EE[ \|\widehat{V}_{\mu,t}^{-1/4}m_{\mu,t}\|^2] +  \EE[ \|\widehat{V}_{\mu,t-1}^{-1/4} m_{\mu,t-1}\|^2 ] \big)\\
& \leq & \bar{\rho} \sum_{t=0}^{T} \alpha_t^2 C_{1,\mu,t},
\end{eqnarray}
where $C_{1,\mu ,t}$ and $C_{2,\mu,T}$ are defined in \eqref{eqn:c2}.

Further, by Lemma~\ref{M2}, the function $f_{\mu}$ is weakly convex. Thus, \eqref{eqn:ex_G_mu} in conjunction with Lemma~\ref{lem:weak:hyp}(ii) yields
 \begin{eqnarray}\label{fmu}
\nonumber
 \EE [  \dotp{\bar x_t - x_t , \gamma_{\mu,t} }]  &=& \dotp{x_t-\bar x_t  , \nabla f_{\mu}(x_t) }
\\ \nonumber &\geq& f_{\mu}(x_t)-f_{\mu}(\bar x_t)- \frac{\rho}{2}\|\widehat{V}_{\mu,t}^{1/4}(x_t - \bar x_t)\|^2 \\
\nonumber
   & =&  ( f(\bar{x}_t)- f_{\mu}(\bar{x}_t))  - (f(x_t)-f_{\mu}({x}_t)) \\\nonumber
   &+& f(x_t) - f(\bar x_t) - \frac{\rho}{2}\|\widehat{V}_{\mu,t}^{1/4}(x_t - \bar x_t)\|^2\\
\nonumber
   & \geq& - | f(\bar{x}_t)- f_{\mu}(\bar{x}_t)|  - |f(x_t)-f_{\mu}({x}_t)| \\
   &+& f(x_t) - f(\bar x_t) - \frac{\rho}{2}\|\widehat{V}_{\mu,t}^{1/4}(x_t - \bar x_t)\|^2.
\end{eqnarray}

Substituting \eqref{ppz}--\eqref{fmu} into \eqref{cx}, utilizing the inequality $\psi^*\leq \psi(x_{T+1})$ and rearranging the terms, we obtain
\begin{eqnarray}\label{eqn:main_ineqsz}
&& \frac{1}{ \sum_{t=0}^{T}\alpha_t}\sum_{t=0}^T (1-\beta_{1,t})\alpha_t \EE[f(x_t) - f(\bar x_t) - \frac{\rho}{2}\|\widehat{V}_{\mu,t}^{1/4}(x_t - \bar x_t)\|^2] \notag\\
& \leq &  \frac{1}{\bar \rho \sum_{t=0}^{T}\alpha_t}\Big(\psi_{1/\bar \rho,\widehat{V}_{\mu,-1}^{1/2}}(x_{0}) - \psi^*
+ \bar \rho  \sum_{t=0}^{T} \alpha_t^2 C_{1,\mu,t} +{\bar \rho}C_{2,\mu,T} \Big) +  C_{3,\mu,T},
\end{eqnarray}
where $ C_{3,\mu,T}$ is defined in \eqref{eqn:c2}.

Now, by the same argument as in \eqref{klks}, we get
\begin{eqnarray}\label{zerooklks}
\nonumber
&&(1-\beta_{1,t^*})\left(f(x_{t^*}) - f(\bar x_{t^*}) - \frac{\rho}{2}\|\widehat{V}_{\mu,t^*}^{1/4}(x_{t^*} - \bar x_{t^*})\|^2\right) \\
&\geq &  \frac{ (1-\beta_{1,t^*}) (\bar \rho-\rho)}{\bar \rho^2} \|\nabla \psi_{1/\bar \rho,\widehat{V}_{\mu,t^*}^{1/2}}(x_{t^*})\|^2.
\end{eqnarray}
Substituting \eqref{zerooklks} into the left hand side of \eqref{eqn:main_ineqsz}, we arrive at the desired result as stated in the theorem.
\end{proof}

Compared to the convergence rate of \textsc{Fema}, Theorem~\ref{thm:stochastic_subzero} exhibits
additional error $C_{3,\mu,T}$ due to the use of zeroth order  gradient estimates. Roughly speaking, if we choose the smoothing parameter $\mu$ reasonably small, then the error $C_{3,\mu,T}$ would reduce, leading to non-dominant effect on the convergence rate of projected \textsc{Zema}. For other terms containing $C_{1,\mu,t}$ and $C_{2,\mu,T}$, the zeroth order gradient estimates are more involved, relying on $\|g_{\mu,k}\|_1$ and $ \|\widehat{\upsilon}_{\mu,t}^{1/2} \|_1$. To bound these terms, we introduce the following assumption on the the function  $F(\cdot,\xi)$.
\begin{assum}\label{it5}
$F(x, \xi)$ is $L_F$-Lipschitz continuous, i.e., for all $x,y \in \dom h$,
$$
|F(x,\xi) - F(y,\xi) | \leq L_F \|x-y\|, \qquad \forall \xi \in U.
$$
\end{assum}

\begin{lem}\label{lem16}
Let $f_{\mu}(x)$ be the smoothing function  defined in \eqref{rand_smooth_func}. Then, under Assumption~\ref{it5}, we have
\begin{eqnarray*}
   |f(x)-f_{\mu}(x) |&\leq & \mu L_F. \label{ww} 
\end{eqnarray*}
\end{lem}
\begin{proof}
Let  $\S  \equiv  (\xi,u)$. Then, we have 
\begin{equation}
  |f(x)-f_{\mu}(x) | \leq  \EE_{\S} \big[|F(x,\xi) - F(x+\mu  u,\xi) |\big] \leq \mu L_F . 
\end{equation}
\end{proof}

We now provide convergence rate of projected \textsc{Zema}.

\begin{cor}\label{cor:complexity}
Under Assumption~\ref{it5} and the same conditions as in Theorem~\ref{thm:stochastic_subzero}, for all $t\in (T]$ let the parameters
be set to
\begin{align*}
\beta_{1,t}&=\beta_1\pi^{t-1},~~~ \pi\in (0,1), ~~~ \bar \rho=2\rho, ~~~\mu = \frac{d}{\sqrt{T+1}}, \quad \textnormal{and} \quad \alpha_t =\frac{\alpha}{\sqrt{T+1}},
\end{align*}
for some $\alpha>0$. Then, for $x_{t^*}$ generated by Algorithm~\ref{alg:subgradientzero}, we have
\begin{equation*}
\EE \left[\|\nabla \psi_{1/(2\rho),{\widehat{V}}_{\mu,t^*}^{1/2}}(x_{t^*})\|^2\right] \leq O\Big(\frac{d^2}{\sqrt{T}}\Big).
\end{equation*}
\end{cor}
\begin{proof}
We combine Lemma~\ref{lem16} with $C_{3,\mu,T}$ defined in \eqref{eqn:c2} to obtain
\begin{eqnarray}\label{fff6}
 \nonumber C_{3,\mu,T} &=&\EE \big[  \max_{t \in (T]}|f(\bar{x}_t)-f_{\mu}(\bar{x}_t)| \big]+\EE \big[  \max_{t \in (T]}|f(x_t)-f_{\mu}(x_t)| \big]\\
 \nonumber
  &\leq& 2 \mu L_F\\
  &=& \frac{2d L_F}{\sqrt{T+1}},
\end{eqnarray}
where the last equality follows since $ \mu = \frac{d}{\sqrt{T+1}}$.

By Assumption~\ref{it5} and \eqref{ppl}, we have
\begin{eqnarray}\label{pffy}
\|G_\mu(x,\xi,u) \|_{\infty} &=\frac{d}{\mu}\|[F(x,\xi)-F(x+\mu u,\xi)]u  \|_{\infty} \leq \tilde{G}_{\infty},
\end{eqnarray}
where $\tilde{G}_{\infty} =  d L_F $.\\
Now, substituting $\beta_{1,t}=\beta_1\lambda^{t-1}$ in  $C_{2,\mu,T}$ defined in \eqref{eqn:c2}, we obtain
\begin{eqnarray} \label{eqn:bc2}
\nonumber
C_{2,\mu,T} &=& \frac{D_{\infty}^2}{2} \big(\sum_{t=0}^T  \beta_1^2 \pi^{2(t-1)} \EE[\|\widehat{\upsilon}_{\mu,t}^{1/2}\|_1] +  \EE[\|\widehat{\upsilon}_{\mu,T}^{1/2}\|_1]\big)\\
\nonumber
&\leq& \frac{dD_{\infty}^2 \tilde{G}_{\infty}}{2} \big(\sum_{t=0}^T  \beta_1^2 \pi^{2(t-1)}  + 1 \big)\\
&\leq& \frac{dD_{\infty}^2 \tilde{G}_{\infty}}{2} \big(\frac{\beta_1^2}{1-2\pi} +1 \big)=: C_2,
\end{eqnarray}	
where the first inequality follows from \eqref{pffy}.

Further, for $C_{1,\mu,t}$ defined in \eqref{eqn:c2} and using \eqref{pffy}, we have that
\begin{eqnarray}\label{eqn:bc1}
\nonumber
  C_{1,\mu,t} &= & \frac{  \sum_{k= 0}^t \tau^{t-k} \EE [\|{g_{\mu,k}}\|_1]}{ (1-\beta_1)\sqrt{(1-\beta_2)(1-\beta_3)} } \\ \nonumber
&\leq&   \frac{  d  \tilde{G}_{\infty} \sum_{k= 0}^t \tau^{t-k} }{(1-\beta_1)\sqrt{(1-\beta_2)(1-\beta_3)} }
\\
&\leq& \frac{   d  \tilde{G}_{\infty} }{(1-\tau) (1-\beta_1)\sqrt{(1-\beta_2)(1-\beta_3)} }=:C_1,
\end{eqnarray}
where  the last inequality uses $$\sum_{t=0}^{T} \tau^t \leq \frac{1}{1-\tau}.$$
Now, substituting \eqref{fff6}, \eqref{eqn:bc2} and \eqref{eqn:bc1} into \eqref{thm}, using $\alpha_t=\frac{\alpha}{\sqrt{T+1}}$, and $\bar \rho=2\rho$, we get
\begin{eqnarray}\label{Projec}
\nonumber
\EE \left[\|\nabla \psi_{1/(2 \rho),\widehat{V}_{\mu,t^*}^{1/2}}(x_{t^*})\|^2\right] &\leq &  \frac{ 2\rho \Delta_\psi+4\rho^2 (\alpha^2 C_1+C_2) }{\rho (1-\beta_1)\sqrt{T+1}\alpha}\\
&+&\frac{2 d L_F }{\sqrt{T+1}}.
\end{eqnarray}
Minimizing the right-hand side of \eqref{Projec}
w.r.t. $\alpha$ yields the choice $\alpha=\sqrt{\frac{\Delta_\psi +2 \rho C_2}{2\rho C_1}}$. Thus,
\begin{eqnarray*}
\nonumber\EE \left[\|\nabla \psi_{1/(2 \rho),\widehat{V}_{\mu,t^*}^{1/2}}(x_{t^*})\|^2\right] & \leq  &  \frac{ 4\sqrt{2 \rho C_1}\sqrt{ \Delta_\psi+2 \rho C_2}}{ (1-\beta_1)\sqrt{T+1}} + \frac{2 d L_F }{\sqrt{T+1}}\\
&= & O\Big(\frac{d^2}{\sqrt{T}} \Big).
\end{eqnarray*}
This completes the proof.
\end{proof}

\begin{rem}
The convergence of \textsc{Zema} achieves the similar rate to \cite{ghadimi2016mini}, which also implies zeroth order projected SGD but ours use adaptive learning rates and momentum techniques. In addition, our theoretical analysis does not necessary need batching and/or smoothness of the function $f$.
\end{rem}

%

\begin{thm}[\textbf{proximal \textsc{Zema}}]\label{thm:stochastic_sub2zero}
Suppose Assumptions~\ref{it1}--\ref{it2} hold and $\|x-y\|_{\infty} \leq D_{\infty}$ for all $x,y \in \dom h$. Further, for all $ t \in (T]$, let
\begin{align*}
& 0 \leq \beta_{1,t}\leq\beta_1, && \{\beta_i\}_{i=1}^3  \in [0,1),   &\tau = \frac{\beta_1}{\sqrt{\beta_2}}<1, \\
& \bar \rho \in \big(\rho,2\rho\big],  && \mu>0, & 0<\alpha_t \leq \frac{1}{\bar \rho}.
\end{align*}
Then, for $x_{t^*}$ generated by Algorithm~\ref{alg:subgradientzero}, we have
\begin{eqnarray}\label{cc}
\nonumber
\EE \left[\|\nabla \psi_{1/\bar \rho,\widehat{V}_{\mu,t^*}^{1/2}}(x_{t^*})\|^2\right] &\leq&  \frac{ \bar \rho \Delta_{\psi} + \bar \rho^2 \big( \sum_{t=0}^{T} \alpha_t^2 C_{1,\mu,t}+  C_{2,\mu,T} \big)}{(\bar \rho-\rho)\sum_{t=0}^T \alpha_t }\\
&+&\frac{\bar{\rho}C_{3,\mu,T}}{\bar{\rho}-\rho}.
\end{eqnarray}
Here, $\Delta_{\psi}= \psi_{1/\bar \rho,\widehat{V}_{\mu,-1}^{1/2}}(x_{0}) - \psi^*$,
\begin{align}\label{zprox:pj}
\nonumber
C_{1,\mu,t} &= \frac{2 \sum_{k=0}^t\tau^{t-k}\EE [\|g_{\mu,k}\|_1]}{(1-\beta_1)\sqrt{(1-\beta_2)(1-\beta_3)}} +\frac{d\tilde{G}_{\infty}}{\sqrt{(1-\beta_2)(1-\beta_3)}}+ \frac{d \tilde{G}_{\infty}^2}{\lambda_{\min}(\widehat{V}_{\mu,t}^{1/2})}, \\\nonumber
C_{2,\mu,T} &= \frac{D_\infty^2}{2} (\EE[ \|\widehat{\upsilon}_{\mu,T}^{1/2} \|_1]+\sum_{t=0}^T\beta_{1,t}^2\EE[ \|\widehat{\upsilon}_{\mu,t}^{1/2} \|_1]),\\
C_{3,\mu,T} &=\EE \big[  \max_{t \in (T]}|f(x_t)-f_{\mu}(x_t)| \big]+\EE \big[  \max_{t \in (T]}|f_{\mu}(\bar{x}_t)-f(\bar{x}_t)| \big].
\end{align}
\end{thm}
\begin{proof}
Let $\delta_t=1-\alpha_t\bar  \rho$ and  $\bar \gamma_{t} = \EE_{t}[G(\bar{x}_t,\xi_t)]$. By the same argument as in \eqref{jk}, we have
\begin{eqnarray}\label{jkz}
\nonumber
 && \EE_{\S_t}[  \|\widehat{V}_{\mu,t}^{1/4}(x_{t+1} - \bar x_t)\|^2]\\
\nonumber
	&\leq& 2\alpha_t^2\big( \EE_{\S_t}[\| \widehat{V}_{\mu,t}^{-1/4} \bar \gamma_{t}\|^2] +  \EE_{\S_t}[\| \widehat{V}_{\mu,t}^{-1/4}m_{\mu,t}\|^2]\big) \\	
&+&  2 \delta_t \alpha_t \EE_{\S_t}[\dotp{(\bar x_t-x_t ), m_{\mu,t} - \bar \gamma_{t}}] + \delta_t^2\EE_{\S_t}[\|\widehat{V}_{\mu,t}^{1/4}(x_t - \bar x_t )\|^2].
	\end{eqnarray}
By definition of $m_{\mu,t}$ in Algorithm~\ref{alg:subgradientzero}, we have
\begin{eqnarray*}
\nonumber
&&\EE_{\S_t}[ \dotp{\bar x_t-x_t , m_{\mu,t}- \bar \gamma_{t}}]\\
\nonumber
&=& \EE_{\S_t}[ \dotp{\bar x_t-x_t ,\beta_{1,t} (m_{\mu,t-1}- g_{\mu,t})}] + \EE_{\S_t}[ \dotp{\bar x_t-x_t , g_{\mu,t}- \bar \gamma_{t}}].
\end{eqnarray*}
We multiple and divide the first term by $ \alpha_t^{1/2} \widehat{V}_{\mu,t}^{-1/4}$ and use the Cauchy-Schwarz inequality to obtain
\begin{eqnarray}\label{eq:thm4.2.2z}
\nonumber
&&
\EE_{\S_t}[\dotp{\bar x_t-x_t ,\beta_{1,t} (m_{\mu,t-1}-g_{\mu,t})}]\\
\nonumber
&\leq & \frac{\alpha_t}{2}\EE_{\S_t}[\|\widehat{V}_{\mu,t}^{-1/4}  (m_{\mu,t-1}-g_{\mu,t})\|^2]+ \frac{\beta_{1,t}^2}{2\alpha_t} \EE_{\S_t}[\|\widehat{V}_{\mu,t}^{1/4}(\bar{x}_t-x_t)\|^2] \\
&\leq & \alpha_t\|\widehat{V}_{\mu,t-1}^{-1/4} m_{\mu,t-1}\|^2+\alpha_t\EE_{\S_t}[\|\widehat{V}_{\mu,t}^{-1/4}  g_{\mu,t}\|^2]+ \frac{\beta_{1,t}^2}{2\alpha_t} \EE_{\S_t}[\|\widehat{V}_{\mu,t}^{1/4}(\bar{x}_t-x_t)\|^2],
\end{eqnarray}
where the second inequality follows since $(\widehat{V}_{\mu,t})_{ii}\geq (\widehat{V}_{\mu,t-1})_{ii}$ for all $i\in [d]$.

Note that by Lemmas \ref{M2} and \ref{lem:weak:hyp}(ii), for $x_t$, $\bar x_t\in\Bbb{R}^d$ with $\varpi\in \nabla f_{\mu}(x_t)$ and $\omega\in \partial f(\bar x_t)$, we have
\begin{eqnarray*}\label{dds}
  f_{\mu}(\bar x_t) &\geq & f_{\mu}(x_t)+\langle \varpi,\bar x_t-x_t \rangle-\frac{\rho}{2}\|\widehat{V}_{\mu,t}^{1/4}(\bar x_t-x_t) \|^2, \quad \text{and}\\
  f(x_t) &\geq & f(\bar x_t)+\langle \omega,x_t-\bar x_t \rangle-\frac{\rho}{2}\|\widehat{V}_{\mu,t}^{1/4}(x_t-\bar x_t) \|^2.
\end{eqnarray*}
Thus,
\begin{align}\label{wqq}
\nonumber\langle \varpi-\omega ,\bar x_t-x_t \rangle &\leq f_{\mu}(\bar x_t)-f_{\mu}(x_t)+f(x_t)-f(\bar x_t)+\rho\|\widehat{V}_{\mu,t}^{1/4}(x_t-\bar x_t) \|^2\\
&\leq    |f(x_t)-f_{\mu}(x_t)|+|f_{\mu}(\bar x_t)-f(\bar x_t)| +\rho\|\widehat{V}_{\mu,t}^{1/4}(x_t-\bar x_t) \|^2.
\end{align}
Then, \eqref{wqq} yields
\begin{align}\label{fdg}
\nonumber &\EE_{\S_t}[ \dotp{\bar x_t-x_t , g_{\mu,t}- \bar \gamma_{t}}]=\dotp{\bar x_t-x_t , \nabla f_{\mu}(x_t)- p}
\\&\leq    |f(x_t)-f_{\mu}(x_t)|+|f_{\mu}(\bar{x}_t)-f(\bar{x}_t)| +\rho\|\widehat{V}_{\mu,t}^{1/4}(x_t-\bar x_t) \|^2,
\end{align}
where $p\in \partial f(\bar{x}_t)$ and the first equality is by \eqref{eqn:ex_G_mu}.

Now, substituting \eqref{eq:thm4.2.2z} and \eqref{fdg} into \eqref{jkz} and using our assumption $\delta_t\leq 1$, we obtain
\begin{eqnarray}\label{befzz2}
\nonumber
 && \EE_{\S_t}[\|\widehat{V}_{\mu,t}^{1/4}(x_{t+1} - \bar x_t)\|^2]\\
\nonumber &\leq&  \|\widehat{V}_{\mu,t-1}^{1/4}(\bar{x}_t-x_t)\|^2
- \big(2\alpha_t(\bar \rho -\rho)+\alpha_t^2 \bar \rho (2\rho-\bar\rho)\big) \EE_{\S_t}[\|\widehat{V}_{\mu,t}^{1/4}(x_t - \bar x_t )\|^2]
\\  \nonumber
 &+&\EE_{\S_t}[ \underbrace{ \|(\widehat{V}_{\mu,t}^{1/4}-\widehat{V}_{\mu,t-1}^{1/4})( \bar x_t-x_{t})\|^2+\beta_{1,t}^2\|\widehat{V}_{\mu,t}^{1/4}( \bar x_t-x_{t})\|^2}_{C_{1,2,\mu,t}}]\\
&+&
\alpha_t^2 \big( \EE_{\S_t}[ \underbrace{2\| \widehat{V}_{\mu,t}^{-1/4}m_{\mu,t}\|^2+2
\|\widehat{V}_{\mu,t-1}^{-1/4}  m_{\mu,t-1}\|^2+2
\| \widehat{V}_{\mu,t}^{-1/4} \bar \gamma_{\mu,t}\|^2 + 2 \|\widehat{V}_{\mu,t}^{-1/4}  g_{\mu,t}\|^2}_{C_{1,1,\mu,t}}]
\big) \nonumber \\
\nonumber &+&  2\alpha_t \EE_{\S_t}[ |f(x_t)-f_{\mu}(x_t)| ]+2\alpha_t \EE_{\S_t}[ |f_{\mu}(\bar{x}_t)-f(\bar{x}_t)| ]  \\\nonumber
 &\leq&   \|\widehat{V}_{\mu,t-1}^{1/4}(\bar{x}_t-x_t)\|^2
- 2\alpha_t\big(\bar \rho-\rho\big) \EE_{\S_t}[ \|\widehat{V}_{\mu,t}^{1/4}(x_t - \bar x_t) \|^2 ]
 +\alpha_t^2 \EE_{\S_t}[C_{1,1,\mu,t}]
 \\&+& \EE_{\S_t}[C_{1,2,\mu,t}]+2\alpha_t \EE_{\S_t}[ |f(x_t)-f_{\mu}(x_t)| ]+2\alpha_t \EE_{\S_t}[ |f_{\mu}(\bar{x}_t)-f(\bar{x}_t)| ] ,
\end{eqnarray}
where the first inequality uses \eqref{eq:thm0012} and the second inequality follows since $\rho<\bar \rho \leq 2\rho$.

Now, using Assumption~\ref{it1}, for the sequence $x_t$ generated by Algorithm~\ref{alg:subgradientzero}, we have
\begin{eqnarray*}
\nonumber
&&\EE_{\S_t}[ \psi_{1/\bar \rho,\widehat{V}_{\mu,t}^{1/2}}(x_{t+1})] \leq  \EE_{\S_t}[ \psi(\bar x_t) + \frac{\bar \rho}{2} \|\widehat{V}_{\mu,t}^{1/4}(x_{t+1} - \bar x_t)\|^2]  \\
\nonumber
&\leq & \psi_{1/\bar \rho,\widehat{V}_{\mu,t-1}^{1/2}}(x_{t})
 + \frac{\bar \rho}{2}\Big( - 2\alpha_t\big(\bar \rho-\rho\big) \EE_{\S_t}[ \|\widehat{V}_{\mu,t}^{1/4}(x_t - \bar x_t) \|^2 ] +\alpha_t^2  \EE_{\S_t}[C_{1,1,\mu,t}]
 \\
 &+&  \EE_{\S_t} [C_{1,2,\mu,t}]+ 2\alpha_t \EE_{\S_t}[ |f(x_t)-f_{\mu}(x_t)| ]+ 2\alpha_t \EE_{\S_t}[ |f_{\mu}(\bar{x}_t)-f(\bar{x}_t)| ] \Big),
\end{eqnarray*}
where the second inequality is by \eqref{befzz2}. Taking expectations on both sides of the above inequality w.r.t. $\S_0,\S_1,\ldots,\S_{t-1}$, and using the law of total expectation, we obtain
\begin{eqnarray}\label{pkpp}
\nonumber
&& \bar \rho\alpha_t\big(\bar \rho-\rho\big) \sum_{t=0}^{T}\EE [\|\widehat{V}_{\mu,t}^{1/4}(x_t - \bar x_t) \|^2
]
\leq   \psi_{1/\bar \rho,\widehat{V}_{\mu,-1}^{1/2}}(x_{0}) - \EE [\psi_{1/\bar \rho,\widehat{V}_{\mu,T}^{1/2}}(x_{T+1})] \\&+&  \frac{\bar \rho }{2}\sum_{t=0}^{T}\alpha_t^2 \EE[C_{1,1,\mu,t}]+\frac{\bar \rho }{2}\sum_{t=0}^{T}\EE[C_{1,2,\mu,t}]+\bar{\rho} C_{3,\mu,T} \sum_{t=0}^T\alpha_t ),
\end{eqnarray}
where $C_{3,\mu,T}$ is defined in \eqref{zprox:pj}. Next, we upper bound each term on the R.H.S. of the above inequality.

Observe that
\begin{eqnarray*}
\nonumber
\frac{\bar \rho }{2}\sum_{t=0}^{T}\alpha_t^2 \EE[C_{1,1,\mu,t}] &\leq &   \frac{\bar{\rho}}{2}  \sum_{t=0}^{T}   \frac{  4 \alpha_t^2 \sum_{k=0}^t{\tau}^{t-k} \EE[ \|{g_{\mu,k}}\|_1] }{{(1-\beta_1)\sqrt{(1-\beta_2)(1-\beta_3)}}}
+ \frac{\bar{\rho}}{2}  \sum_{t=0}^{T}   \frac{ 2 \alpha_t^2  \EE[ \|{g_{\mu,t}}\|_1] }{{\sqrt{(1-\beta_2)(1-\beta_3)}}}
\\\nonumber&+&\frac{\bar{\rho}}{2}\sum_{t=0}^T2\alpha_t^2 \EE[\| \widehat{V}_{\mu,t}^{-1/4} \bar \gamma_{\mu,t}\|^2]\\ \nonumber
 &\leq &   \bar{\rho}  \sum_{t=0}^{T}   \frac{  2 \alpha_t^2 \sum_{k=0}^t\tau^{t-k} \EE[ \|{g_{\mu,k}}\|_1] }{(1-\beta_1){\sqrt{(1-\beta_2)(1-\beta_3)}}}+\frac{\bar{\rho}\sum_{t=0}^T\alpha_t^2d\tilde{G}_{\infty}}{\sqrt{(1-\beta_2)(1-\beta_3)}} \\&+&  \frac{\bar{\rho}\sum_{t=0}^{T}\alpha_t^2d \tilde{G}^2_{\infty}}{\lambda_{\min}(\widehat{V}_{\mu,t}^{1/2})}
=  \bar{\rho} \sum_{t=0}^{T} \alpha_t^2 C_{1,\mu,t},
\end{eqnarray*}
where the first inequality follows from Lemma~\ref{lem:spar}; the second inequality uses \eqref{pffy}; and the last equality is obtained from \eqref{zprox:pj}.

Further,  by the same argument as in \eqref{pp}, we have
\begin{eqnarray*}
\nonumber
C_{1,2,\mu,t} & \leq & \frac{\bar \rho}{2}  \sum_{t=0}^T  \big(\EE [\|\widehat{\upsilon}_{\mu,t}^{1/4}-\widehat{\upsilon}_{\mu,t-1}^{1/4} \|_1^2]+\beta^2_{1,t}\EE[\| \widehat{\upsilon}_{\mu,t}^{1/2}\|_1] \big) \|x_t - \bar{x}_t\|_{\infty}^2 \\
& \leq & \frac{\bar \rho D^2_{\infty}}{2} (\EE[\| \widehat{\upsilon}_{\mu,T}^{1/2}\|_1]+\sum_{t=0}^T\beta_{1,t}^2\EE[\| \widehat{\upsilon}_{\mu,t}^{1/2}\|_1])=   \bar{\rho} C_{2,\mu,T},
\end{eqnarray*}
where  $C_{2,\mu,T}$ is defined in \eqref{zprox:pj}.

It follows from Assumption~\ref{it1} that the function $\psi$ is weakly convex. This together with \eqref{grad} implies that
\begin{eqnarray*}
&& \frac{1}{\bar \rho^2}\nonumber \|\nabla \psi_{1/\bar \rho,\widehat{V}_{\mu,t^*}^{1/2}}(x_{t^*})\|^2
 =\EE [\|\widehat V_{\mu,t^*}^{1/4}(x_{t^*} - \bar x_{t^*})\|^2]\\
&=&\frac{1}{\sum_{t=0}^T \alpha_t}\sum_{t=0}^T \alpha_t \EE [\|\widehat V_{\mu,t}^{1/4}(x_{t} - \bar x_{t})\|^2]
\leq  \frac{1}{\bar \rho\big(\bar \rho-\rho\big)\sum_{t=0}^{T}\alpha_t}\\
\nonumber&&\cdot\Big(\psi_{1/\bar \rho,\widehat{V}_{\mu,-1}^{1/2}}(x_{0}) - \psi^*  + \bar{\rho}\sum_{t=0}^{T} \alpha_t^2 C_{1,\mu,t}+ \bar{\rho} C_{2,\mu,T}+\bar{\rho} C_{3,\mu,T}\sum_{t=0}^T \alpha_t \Big),
\end{eqnarray*}
where the last inequality follows since by our assumption $\psi^*$ is finite which implies that $ \psi_{1/\bar \rho,\widehat{V}_{\mu,T}^{1/2}}(x_{T+1})\geq \psi^*$.
\end{proof}

The following corollary shows how to choose $\alpha_t$,  $\bar \rho$ and the smoothing parameter $\mu$ appropriately in the zeroth-order proximal setting.

\begin{cor}\label{cor:da}
Under Assumption \ref{it5} and the same conditions as in Theorem~\ref{thm:stochastic_sub2zero}, for all $t \in (T]$ let the parameters be set to
\begin{align*}
&\beta_{1,t}=\beta_1\pi^{t-1}, && \pi\in (0,1), & \bar \rho=2\rho, \\
&\mu=\frac{d}{\sqrt{T+1}}, && \alpha_t =\frac{\alpha}{\sqrt{T+1}},  & \alpha \in (0,\frac{1}{2\rho}].
\end{align*}
Then, for $x_{t^*}$ returned by Algorithm~\ref{alg:subgradientzero}, we have
\begin{equation*}
\EE \left[\|\nabla \psi_{1/(2\rho),{\widehat{V}}_{\mu,t^*}^{1/2}}(x_{t^*})\|^2\right] \leq O\Big(\frac{d^2}{\sqrt{T}} \Big).
\end{equation*}

\end{cor}
\begin{proof}
Combine Lemma~\ref{lem16} with $C_{3,\mu,T}$ defined in  \eqref{zprox:pj} to obtain
\begin{eqnarray}\label{fff}
\nonumber C_{3,\mu,T}&=&\EE \big[  \max_{t \in (T]}|f(x_t)-f_{\mu}(x_t)| \big]+\EE \big[  \max_{t \in (T]}|f_{\mu}(\bar{x}_t)-f(\bar{x}_t)| \big]\\
\nonumber
& \leq& 2 \mu L_F  \\
&=& \frac{2 d L_F}{\sqrt{T+1}}.
\end{eqnarray}
Similarly,
substituting $\beta_{1,t}=\beta_1\pi^{t-1}$ in  $C_{2,\mu,T}$ defined in \eqref{zprox:pj} and using Lemma~\ref{lm:infty} and \eqref{pffy}, we obtain
\begin{eqnarray} \label{eqn:zproxbc23}
\nonumber
C_{2,\mu,T} &=& \frac{D_\infty^2}{2} (\EE[ \|\widehat{\upsilon}_{\mu,T}^{1/2} \|_1]+\sum_{t=0}^T\beta_{1,t}^2\EE[ \|\widehat{\upsilon}_{\mu,t}^{1/2} \|_1])
 \\\nonumber&\leq& \frac{d D_{\infty}^2 \tilde{G}_{\infty}}{2}(\sum_{t=0}^T\beta_1^2 \pi^{2(t-1)}+1)
 \\&\leq& \frac{d D_{\infty}^2 \tilde{G}_{\infty}}{2}(\frac{\beta_1^2}{1-2\pi}+1)=: C_2,
\end{eqnarray}	
where $\tilde{G}_{\infty}$ is defined in \eqref{pffy}.

Also, for the term $C_{1,\mu,t}$ defined in \eqref{zprox:pj}, we have
\begin{eqnarray}\label{eqn:zproxbc12}
\nonumber
C_{1,\mu,t} &=&  \frac{ 2\sum_{k=0}^t\tau^{t-k} \EE [\|g_{\mu,k}\|_1]}{(1-\beta_1)\sqrt{(1-\beta_2)(1-\beta_3)}}+\frac{ d\tilde{G}_{\infty}}{\sqrt{(1-\beta_2)(1-\beta_3)}} + \frac{ d\tilde{G}_\infty^2}{\lambda_{\min}(\widehat{V}_{\mu,t}^{1/2})} \\
\nonumber
&\leq&  (\frac{2 \sum_{k= 0}^t \tau^{t-k}  }{(1-\beta_1) }+1)\frac{ d\tilde{G}_{\infty}}{\sqrt{(1-\beta_2)(1-\beta_3)}}  +\frac{ d \tilde{G}_\infty^2}{\lambda_{\min}(\widehat{V}_{\mu,t}^{1/2})}\\
 &\leq &  (\frac{2 }{ (1-\tau)(1-\beta_1)  }+1)\frac{ d\tilde{G}_{\infty}}{\sqrt{(1-\beta_2)(1-\beta_3)}}  +\frac{ d\tilde{G}_\infty^2}{\lambda_{\min}(Q)} =:C_1,
\end{eqnarray}
where the first inequality is obtained from Lemma~\ref{lm:infty} and \eqref{pffy}; the last inequality uses  $$\sum_{t=0}^{T} \tau^t \leq \frac{1}{1-\tau}, \qquad \text{where} \quad  \tau = \beta_1/\sqrt{\beta_2}<1. $$
Now, substituting \eqref{fff}-\eqref{eqn:zproxbc12} into \eqref{cc} and using $\alpha_t=\frac{\alpha}{\sqrt{T+1}}$, and $\bar \rho=2\rho$, we get
\begin{equation*}
\EE \left[\|\nabla \psi_{1/(2 \rho),\widehat{V}_{\mu,t^*}^{1/2}}(x_{t^*})\|^2\right] \leq  \frac{ 2\rho \Delta_\psi+4\rho^2 (\alpha^2 C_1+C_2) }{\rho \sqrt{T+1}\alpha} + \frac{2d L_F}{\sqrt{T+1}}.
\end{equation*}
Minimizing the right-hand side of above inequality
w.r.t. $\alpha$ yields the choice $\alpha=\sqrt{\frac{ \Delta_\psi +2\rho C_2}{2\rho C_1}}$. Thus,
\begin{eqnarray*}
\nonumber\EE \left[\|\nabla \psi_{1/(2 \rho),\widehat{V}_{\mu,t^*}^{1/2}}(x_{t^*})\|^2\right] & \leq  &  \frac{ 4\sqrt{2 \rho C_1}\sqrt{ \Delta_\psi+2\rho C_2}}{ \sqrt{T+1}} + \frac{2 d L_F}{\sqrt{T+1}}
\\
&= & O\Big(\frac{d^2}{\sqrt{T}} \Big).
\end{eqnarray*}
\end{proof}

\section{Experiments}\label{Experiments}

Based on the insights gained from our theorems, we now present empirical results on both synthetic and real-world data sets. For our experiments, we study the problem of robust phase retrieval \cite{shechtman2015phase} and neural networks with Exponentiated Linear Units (ELUs) activation functions \cite{clevert2015fast}, representing weakly convex settings. The performance of the proposed EMA-type algorithms, i.e., \textsc{Fema} and \textsc{Zema} are compared with the model-based \textsc{Sgd} \cite{duchi2018stochastic,davis2019stochastic} and its zeroth order variant denoted by \textsc{Z-Sgd}.

All algorithms have been implemented and run on a Mac machine equipped with a 1.8 GHz Intel Core i5 processor and 8 GB 1600 MHz DDR3 memory.

\subsection{Robust phase retrieval}

Given a set of tuples $\{(a_i,b_i)\}_{i=1}^n\subset\RR^d\times \RR$, we wish to find a vector $x\in \RR^d$ satisfying
\begin{equation*}
(a_i^\top x)^2=b_i, \quad \text{for \quad $i\in[n]$}.
\end{equation*}
This problem naturally arises in a number of real-world situations, including phase retrieval \cite{fienup1982phase,fienup1987reconstruction} and is a combinatorial problem that is, in the worst case, NP-hard \cite{fickus2014phase}. However, when the set of measurements $\{b_i\}_{i=1}^n$ is corrupted by some gross outliers, one may consider the following \textit{robust phase retrieval} objective \cite{shechtman2015phase}:
\begin{align}  \label{eqn:robust-pr}
 \min_x ~ f(x) = \frac{1}{n}\sum_{i=1}^n \underbrace{|\dotp{a_i, x}^2 - b_i|}_{:=f_{i}(x)},
\end{align}
where $f_{i}$ is a $2 \|a_i\|_2^2$-weakly convex and nonsmooth function.

In our implementation, all the datasets are generated according to the following procedure: (i) We generate standard Gaussian measurements $a_i\sim N(0,I_{d\times d})$,  for $i\in[n]$; (ii)  generate the target signal $x^*$ and initial point $x_0$ uniformly on the unit sphere; and (iii) set $b_i = \langle a_i, x^*\rangle^2$ for each $i\in[n]$. The parameters of \textsc{Fema} and \textsc{Zema} are set to: \textsc{Fema1} ($\beta_{1,t} = 0.9$, $\beta_2=\beta_3=0$); \textsc{Fema2} ($\beta_{1,t} = 0.9$, $\beta_2= 0.999$, $\beta_3 =0$); and \textsc{Fema3} ($\beta_{1,t} =\beta_3 = 0.9$, $\beta_2= 0.999$). The smoothing parameter for \textsc{Zema} is set to $\mu =10/\sqrt{ (T+1)}$.

\begin{figure}[t]
	\centering 
\begin{tabular}{c}
\includegraphics[scale=0.14]{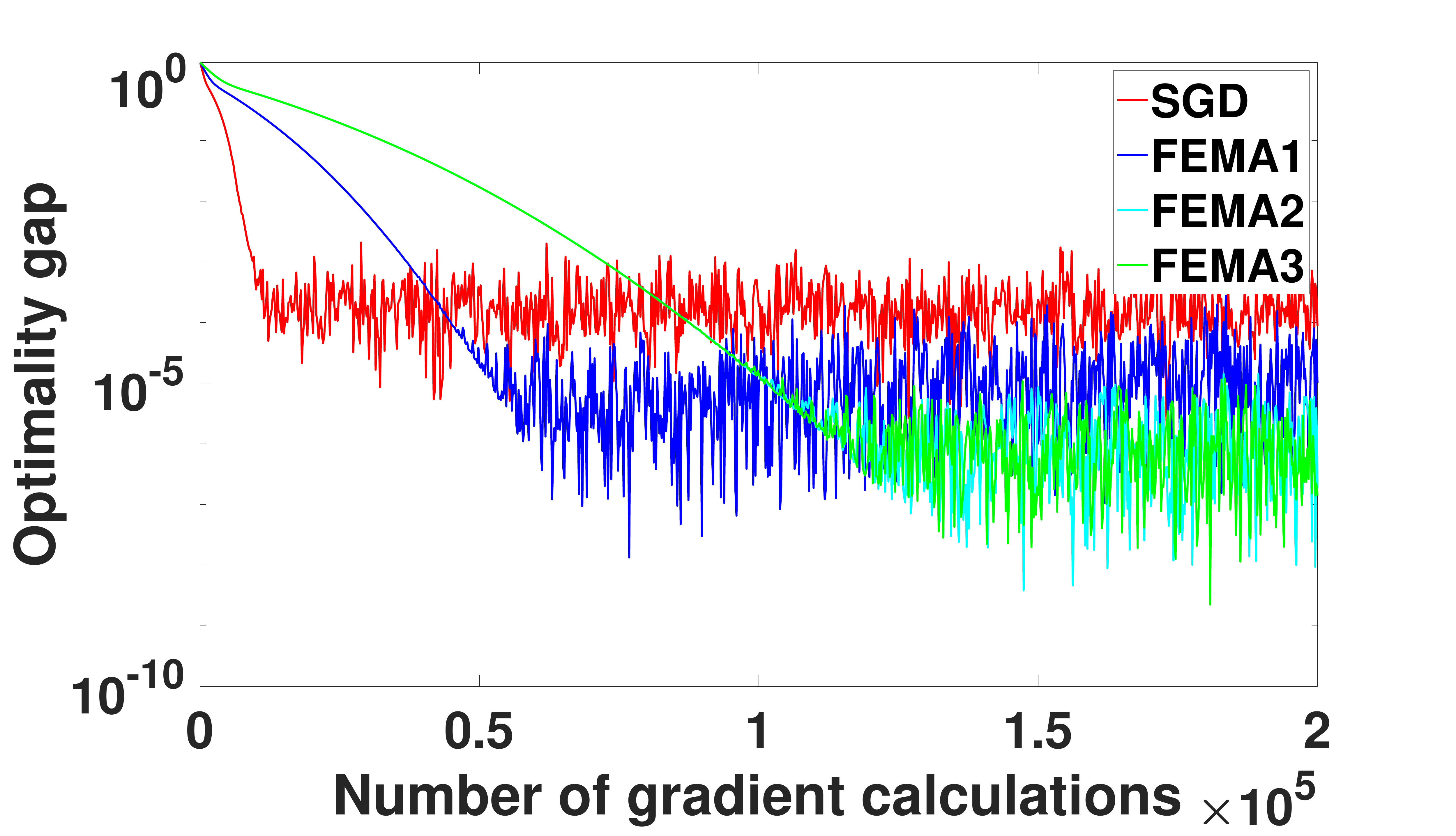}
\includegraphics[scale=0.14]{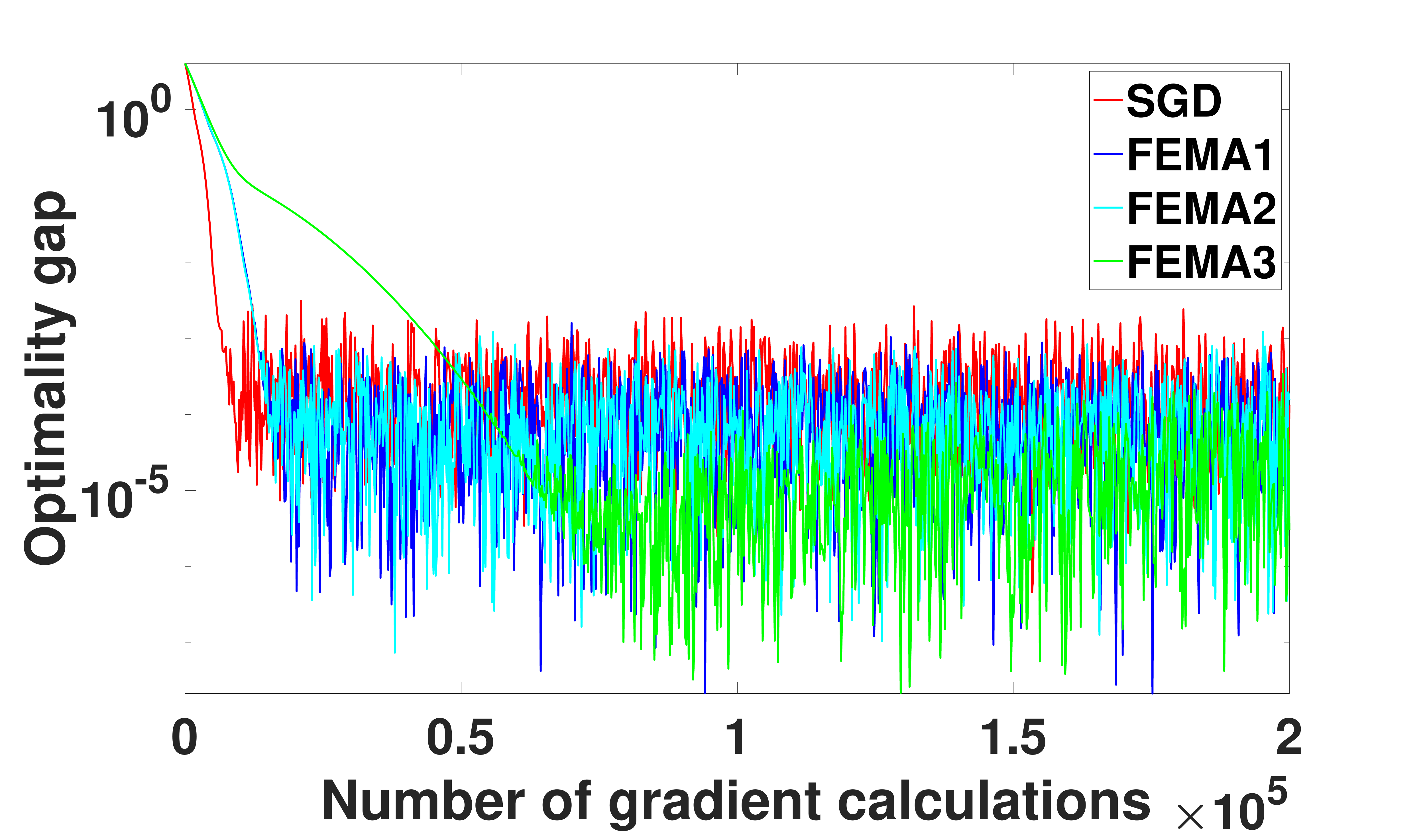}
\\
(a) Stochastic first-order projected (left) and  proximal (right) methods.
\\
\includegraphics[scale=0.14]{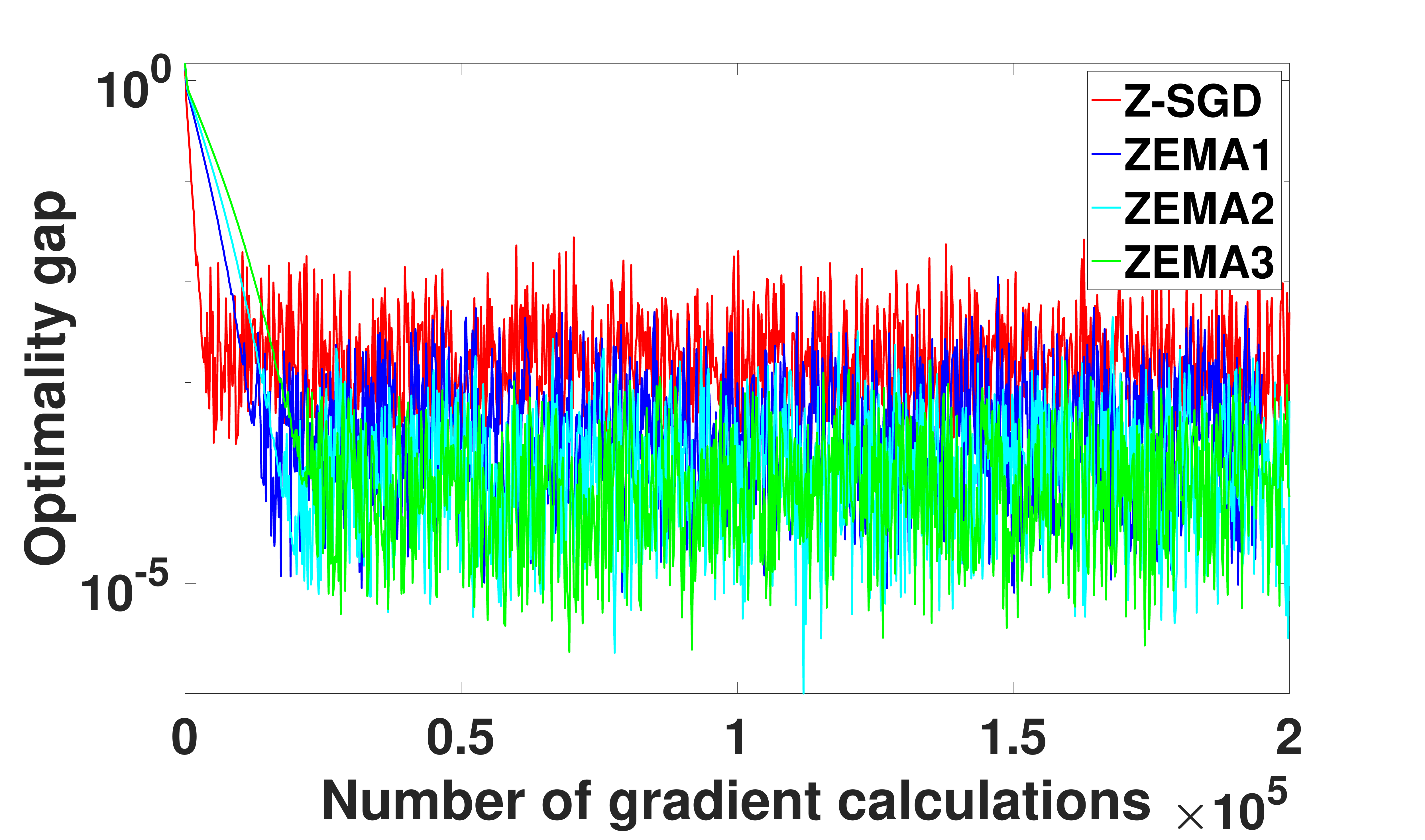}
\includegraphics[scale=0.14]{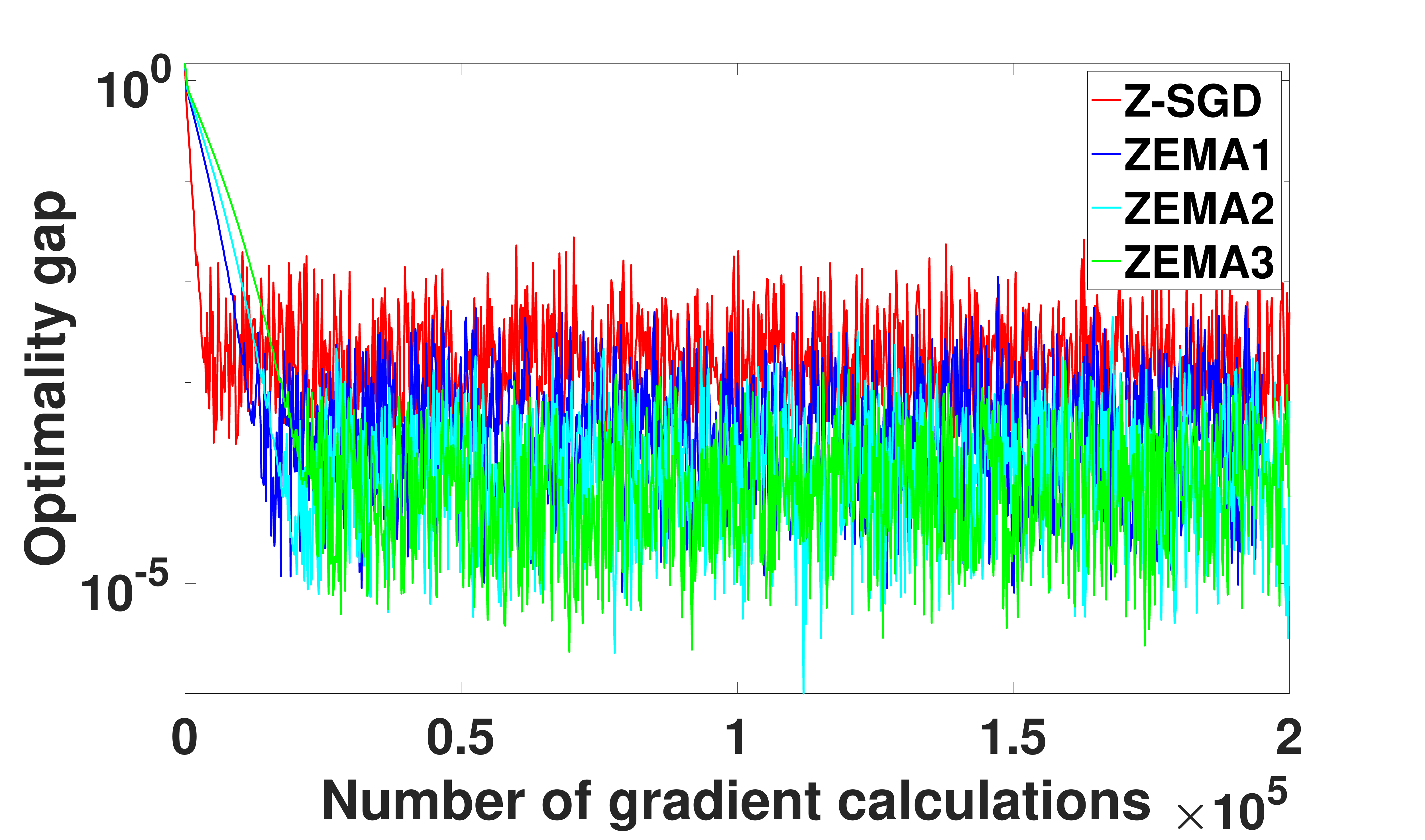}
\\
(b) Stochastic zeroth-order projected (left) and  proximal (right) methods.
 \\
\end{tabular}
\caption{Convergence of different stochastic algorithms over 500 epochs for the robust phase retrieval problem.}\label{fig:phase}
\end{figure}

We illustrate on the stochastic projected and proximal subgradient methods. We examine both cases with $d=10$ and $n= 1000$ and record the result in Figure~\ref{fig:phase}. In each set of experiments, we use $10$ equally spaced stepsize parameters $\alpha$ between $0.0001$ and $0.1$. The experiment is repeated ten times and the average accuracy and performance are depicted in Figure~\ref{fig:phase}. It is clear from the panels in Figure \ref{fig:phase} that \textsc{Fema} and \textsc{Zema} outperform \textsc{Sgd} and \textsc{Z-Sgd} algorithms in terms of the optimality gap $\Delta_{\psi}=\psi_{1/\bar \rho,\widehat{V}_{-1}^{1/2}}(x_{0}) - \psi^*$.

\subsection{Neural networks}\label{section:NN}

We present empirical results showcasing the important aspect of our framework--adaptiveness of learning rates. We test the performance of our models for training neural networks for classification tasks using the CIFAR10 and CIFAR100 datasets (60,000 $32 \times 32$ images, 10 and 100 classes, respectively) \cite{krizhevsky2009learning}. For our experiment, we use \textsc{Fema} to train ResNets110 \cite{he2016deep}, which is very popular architecture, producing state-of-the-art results across many computer vision tasks. We adopt a standard data augmentation scheme (mirroring/shifting) that is widely used.

\begin{table}[!htb]
    \begin{minipage}{.5\linewidth}
    \small{
      \centering
\begin{tabular}{lcccccc}
    \hline
Method   &  $\alpha$ & CIFAR-10  &CIFAR-100\\
    \hline
        \hline
\textsc{Sgd}   & 1e-3 & 6.89 (6.11)   & 26.81 (25.31)   \\
\textsc{Fema1}  & 1e-3 & 5.67 (5.41) & 26.59 (25.08) \\
\textsc{Fema2}  & 1e-3  &5.61(5.36)&  25.83 (25.50)\\
\textsc{Fema3}   & 1e-3 &\textbf{5.51 (5.34)} & \textbf{25.02 (24.56})\\
\hline \\
\end{tabular}
}    
    \end{minipage}%
    \begin{minipage}{.5\linewidth}
      \small{     
      \centering
\begin{tabular}{cccc}
    \hline
Method   &  $\alpha$  & CIFAR-100& CIFAR-100  \\
    \hline
        \hline
\textsc{Fema3}   &  0.0001 & 6.51 (6.00)  &   26.33 (25.49)  \\
      &  0.001    & \textbf{5.51 (5.34)} & \textbf{25.02 (24.56}) \\
    &  0.005   &  5.73 (5.60)  &   25.68 (25.17) \\
   &  0.01 & 6.74 (6.12)  &   26.20 (25.24) \\
\hline \\
\end{tabular}
      }
    \end{minipage}
        \caption{ResNets110 test error (in \%) on CIFAR-10  and CIFAR-100 datasets. All our experiments were run for 500 epochs with batch-size of 256. We report the mean error over the last 10 epochs and minimum error over all epochs (inside parenthesis) of the median error over five tries.}\label{tab:NN}
\end{table}

For our experiment, we use ELUs activation function \cite{clevert2015fast} so that the loss is weakly convex.\footnote{The loss is of composite form $f = h \circ c$, with $h$ convex and $c$ smooth.}. As seen from Table~\ref{tab:NN}, without any tuning, our default parameter setting achieves state-of-the-art results for this network. In particular, we see \textsc{Fema3} consistently outperforming other algorithms, especially on the CIFAR-100 dataset. Further, as shown in Table~\ref{tab:NN} (right), our proposed parameterization $\alpha_t=\alpha=1e-3$ obtained the best accuracy.

\section{Conclusion}\label{conclusion}
In this paper, we examined first and zeroth order adaptive methods for nonconvex \& nonsmooth optimization. We provided some mild sufficient conditions to ensure convergence of a class of adaptive algorithms, which include \textsc{Adam} and  \textsc{RMSprop} as
special cases. To the best of our knowledge, the convergence of adaptive algorithms for nonconvex and nonsmooth problems was unknown before. We also showed empirically on selected settings how adaptive algorithms can perform better than \textsc{Sgd} and its zeroth-order variants.

%
%
\bibliographystyle{abbrv}

\bibliography{ref}

\vspace{4cm}
\begin{center}
{\Large
\textsc{Appendix}
}
\end{center}



\subsubsection*{Proof of Lemma \ref{lem:weak:hyp}}
\begin{proof}
Th proof is similar to that of \cite[Theorem 3.1]{daniilidis2005filling}.  (i) $ \Rightarrow $ (ii). Since $\psi$ is $(\rho,Q)$-weakly convex, it follows from Definition~\ref{def:weak} that the function $\psi_{\rho,Q}(x):=\psi(x) + \frac{\rho}{2}\|Q^{1/2}x\|^2$ is convex. This implies that the subgradient $\partial \psi_{\rho,Q}(x)$ exists and can be computed as follows
$$\partial \psi_{\rho,Q}(x) \in \partial \psi(x)+\rho Qx.$$
Now, let $\varpi\in \partial \psi(x)$. By the convexity of $\psi_{\rho,Q}(x)$, we have
\begin{equation} \label{eqn:wcon}
 \psi_{\rho,Q}(y)\geq \psi_{\rho,Q}(x)+\langle \varpi+\rho Q x,y-x \rangle, \qquad \textnormal{for\, all\,} ~~ x,y\in \Bbb{R}^d.
\end{equation}
By \eqref{eqn:wcon}, we further have
\begin{align*}
\nonumber \psi(y) &\geq \psi(x)+\langle \varpi,y-x \rangle+\frac{\rho}{2}\|Q^{1/2}x\|^2-\frac{\rho}{2}\|Q^{1/2}y\|^2 + \langle \rho Q x,y \rangle - \langle \rho Q x,x \rangle\\
&= \psi(x)+\langle \varpi,y-x \rangle-\frac{\rho}{2}\|Q^{1/2}y\|^2-\frac{\rho}{2}\|Q^{1/2}x\|^2 + \langle \rho Q x,y \rangle \\
 &= \psi(x)+\langle \varpi,y-x \rangle-\frac{\rho}{2}\|Q^{1/2}(y-x)\|^2.
\end{align*}
(ii) $\Rightarrow$ (iii). For any $x,y\in\Bbb{R}^d$ with $\varpi\in \partial \psi(x)$ and $\omega\in \partial \psi(y)$, it follows from \eqref{eqn:stronger_ineq} that
\begin{eqnarray*}\label{dds}
  \psi(y) &\geq & \psi(x)+\langle \varpi,y-x \rangle-\frac{\rho}{2}\|Q^{1/2}(y-x) \|^2, \\
  \psi(x) &\geq & \psi(y)+\langle \omega,x-y \rangle-\frac{\rho}{2}\|Q^{1/2}(x-y) \|^2.
\end{eqnarray*}
Adding the above inequalities gives the desired result.
\\
(iii) $\Rightarrow$ (i). It follows from \eqref{eqn:stronger_ineq_hypo} that
$$\langle \omega-\varpi,x-y \rangle\geq -\rho \|Q^{1/2}(y-x)\|^2,$$
which yields
$$\langle \omega+\rho Q y - (\varpi+\rho Q x),y-x \rangle \geq 0. $$
As a result, the subdifferential of $\psi_{\rho,Q}(\cdot)$ is a globally monotone map.
Applying \cite[Theorem 12.17]{rockafellar2009variational}, we conclude that $ \psi_{\rho,Q}(\cdot)$ is convex.
\end{proof}

\begin{lem}\label{lem:spar}
\textnormal{\cite[Lemma~1]{nazari2019dadam}} Let $ 0 \leq \beta_{1,t}\leq \beta_1$, $\{\beta_i\}_{i=1}^3 \in [0,1)$ and $\tau = \beta_1/\sqrt{\beta_2}<1$. Then, for $m_{t}$ and $ \widehat{V}_t$ generated by Algorithm~\ref{alg:subgradient}, we have
\begin{align*}
 \|\widehat{V}_t^{-1/4} m_{t}\|^2 \leq   \frac{ \sum_{k=0}^t{\tau}^{t-k}  \|{g_{k}}\|_1 }{{(1-\beta_1)\sqrt{(1-\beta_2)(1-\beta_3)}}}.
\end{align*}
\end{lem}

\begin{lem}\textnormal{\cite{nazari2019dadam}}\label{lm:infty}
Suppose Assumption~\ref{it2} holds and $\|G(x, \xi)\|_{\infty} \leq G_{\infty}$ for all $x \in \cX$ and $\xi \in \Omega$.  Then, for $g_t$,$m_t$, and $\widehat{\upsilon}_t$ defined in Algorithm \ref{alg:subgradient}, we have $\|g_t\|_\infty \leq G_\infty$, $\|\widehat{\upsilon}_t^{1/2}\|_\infty \leq G_\infty$ and $\|m_{t}\|_\infty \leq G_\infty$.
\end{lem}

\begin{lem}\label{as} \textnormal {\cite[Lemma 6.3.a]{gao2018information}}
Let $u$  be a $d$-dimensional random vector drawn uniformly from the sphere of a unit ball $B$ . Then,
\begin{equation*}
\frac{1}{\mathcal{V}(d)}\int_{u\in B}\|u\|^pdu=\frac{d}{d+p}, \quad \textnormal{and} \quad \int_{u\in B}u u^\top du=\frac{\mathcal{V}(d)}{d}I.
\end{equation*}
 Here, $\mathcal{V}(d)$ denotes the the volume of the unit ball in $\Bbb{R}^d$ and $I$ is the identity matrix in $\Bbb{R}^{d\times d}$.
\end{lem}

\end{document}